\documentclass[11pt,letterpaper]{article}
\usepackage[utf8]{inputenc}

\usepackage{amsmath, amsthm}
\usepackage{MnSymbol} 
\usepackage{graphicx,color}
\usepackage[hyphens]{url}
\usepackage{dsfont}
\usepackage{booktabs}
\usepackage[normalem]{ulem}
\usepackage{mathtools}
\usepackage{hyperref}
\usepackage[nameinlink,capitalize]{cleveref}
\usepackage{multirow}
\usepackage{algorithm}
\usepackage[noend]{algpseudocode}
\usepackage{tikz} %

\usepackage{soul} %

\usepackage[giveninits=true, maxnames=10, style=alphabetic, sorting=nyt]{biblatex}
\usepackage{subcaption}

\usepackage{tabularx}
\newcolumntype{L}{>{\raggedright\arraybackslash}X}

\DeclareMathAlphabet\mathbb{U}{msb}{m}{n}

\numberwithin{equation}{section}
\newtheorem{theorem}{Theorem}[section]
\newtheorem{cor}[theorem]{Corollary}
\newtheorem{lemma}[theorem]{Lemma}
\newtheorem{remark}[theorem]{Remark}

\newtheorem{obs}[theorem]{Observation}

\newtheorem{example}[theorem]{Example}

\newtheorem{defin}[theorem]{Definition}
\newtheorem{ass}[theorem]{Assumption}

\newcommand{\cD}{\mathcal{D}}

\newcommand{\cM}{\mathcal{M}}
\newcommand{\cN}{\mathcal{N}}

\newcommand{\cP}{\mathcal{P}}

\newcommand{\cX}{\mathcal{X}}

\newcommand{\R}{\mathbb{R}}

\newcommand{\N}{\mathbb{N}}

\newcommand*{\E}{\mathbb{E}}

\providecommand{\tr}{\operatorname{tr}}

\newcommand*{\defeq}{\coloneqq}

\renewcommand{\le}{\leqslant}
\renewcommand{\ge}{\geqslant}
\renewcommand{\leq}{\leqslant}
\renewcommand{\geq}{\geqslant}

\providecommand{\abs}[1]{\lvert{#1}\rvert}

\providecommand{\norm}[1]{\lVert{#1}\rVert}

\newcommand{\sig}{\sigma}

\newcommand{\Hell}{\mathsf{D}}

\newcommand{\KL}{\mathsf{KL}}

\newcommand{\Ren}{\mathsf{R}}

\usepackage{bbm}
\usepackage{blkarray}
\usepackage{cancel}
\usepackage{enumitem}
\usepackage{float}
\usepackage{mathrsfs}
\usepackage{microtype}
\usepackage{ragged2e}
\usepackage{sectsty}
\usepackage{thmtools}
\usepackage[Symbolsmallscale]{upgreek}
\usepackage{array}

\usetikzlibrary{cd}

\hypersetup{citecolor=red, colorlinks=true, linkcolor=blue}

\makeatletter
\patchcmd{\@counteralias}
{\@ifdefinable{c@#1}}
{\expandafter\@ifdefinable\csname c@#1\endcsname}
{}{}
\makeatother

\makeatletter
\newcommand{\opnorm}{\@ifstar\@opnorms\@opnorm}
\newcommand{\@opnorms}[1]{%
	\left|\mkern-1.5mu\left|\mkern-1.5mu\left|
	#1
	\right|\mkern-1.5mu\right|\mkern-1.5mu\right|
}
\newcommand{\@opnorm}[2][]{%
	\mathopen{#1|\mkern-1.5mu#1|\mkern-1.5mu#1|}
	#2
	\mathclose{#1|\mkern-1.5mu#1|\mkern-1.5mu#1|}
}
\makeatother

\newcommand{\PreserveBackslash}[1]{\let\temp=\\#1\let\\=\temp}
\newcolumntype{C}[1]{>{\PreserveBackslash\centering}p{#1}}

\newcommand\bs[1]{\boldsymbol{#1}}

\newcommand\mc[1]{\mathcal{#1}}

\newcommand\ms[1]{\mathscr{#1}}

\newcommand\msf[1]{\mathsf{#1}}

\definecolor{MITBrown}{RGB}{164, 31, 50}

\DeclareMathOperator\Ric{\msf{Ric}}

\DeclareMathOperator\cov{cov}

\DeclareMathOperator\divergence{div}

\DeclareMathOperator\id{id}

\DeclareMathOperator\law{law}

\DeclareMathOperator\var{var}

\renewcommand{\Pr}{\mathbb{P}}

\newcommand{\D}{\mathrm{d}}

\newcommand{\Coup}{\ms C}

\newcommand\deq{\coloneqq}

\newcommand\mmid{\mathbin{\|}}

\newcommand\T{\mathsf{T}}

\newcommand\bigsp{\vphantom{\big|}}
\newcommand\Bigsp{\vphantom{\Big|}}

\newcommand{\essinf}[1]{\mathop{#1\text{-}\mathrm{ess\,inf}}}
\newcommand{\esssup}[1]{\mathop{#1\text{-}\mathrm{ess\,sup}}}

\newcommand{\CD}{\msf{CD}}
\newcommand{\MEKL}{\msf{MY}_{\KL(\cdot \mmid \pi)}}
\newcommand{\PathA}{\bs\mu}
\newcommand{\PathAux}{\bs\mu'}

\usepackage[margin=1in]{geometry}

\makeatletter
\def\blfootnote{\gdef\@thefnmark{}\@footnotetext}
\makeatother

\begin{document}

	\title{Shifted Composition I: Harnack and Reverse Transport Inequalities}

 	\author{
		Jason M. Altschuler\\
		UPenn\\
		\texttt{alts@upenn.edu}
		\and
		Sinho Chewi \\
		IAS \\
		\texttt{schewi@ias.edu}
	}
	\date{}
	\maketitle

	\begin{abstract}
            We formulate a new information-theoretic principle---the \emph{shifted composition rule}---which bounds the divergence (e.g., Kullback--Leibler or R\'enyi) between the laws of two stochastic processes via the introduction of auxiliary shifts.
            In this paper, we apply this principle to prove reverse transport inequalities for diffusions which, by duality, imply F.-Y.\ Wang's celebrated dimension-free Harnack inequalities.
            Our approach bridges continuous-time coupling methods from geometric analysis with the discrete-time shifted divergence technique from differential privacy and sampling.
            It also naturally gives rise to (1) an alternative continuous-time coupling method based on optimal transport, which bypasses Girsanov transformations, (2) functional inequalities for discrete-time processes, and (3) a family of ``reverse'' Harnack inequalities.
	\end{abstract}

    % \footnotesize
	\setcounter{tocdepth}{2}
	\tableofcontents	
	\normalsize

    % !TEX root = ../scp1.tex

\section{Introduction}\label{sec:intro}

In this paper, we formulate a new technique for bounding information-theoretic divergences, such as the Kullback--Leibler (KL) or R\'enyi divergence, between two probability laws.
In the case of the KL divergence, it extends the classical \emph{chain rule}
\begin{align}\label{eq:kl_chain_rule}
    \KL(\bs\mu^Y \mmid \bs \nu^Y) \le \KL(\bs \mu^{X,Y} \mmid \bs \nu^{X,Y})
    &= \KL(\bs \mu^X \mmid \bs \nu^X) + \int \KL(\bs \mu^{Y\mid X=x} \mmid \bs \nu^{Y\mid X=x}) \, \bs \mu^X(\D x)\,.
\end{align}
Here $X$ and $Y$ are jointly defined random variables on a suitable probability space $\Omega$, $\bs\mu$ and $\bs\nu$ are two probability measures over $\Omega$, and we use the obvious notation (e.g., $\bs\mu^{X,Y}$ denotes the joint law of $(X,Y)$, $\bs \mu^X$ denotes the marginal law of $X$, and $\bs \mu^{Y\mid X=x}$ denotes the conditional law of $Y$ given $X=x$, all under the measure $\bs\mu$).
The first inequality in~\eqref{eq:kl_chain_rule} follows from the data-processing inequality (see Theorem~\ref{thm:renyi_prop}).

Our technique is based on a simple yet crucial modification of~\eqref{eq:kl_chain_rule}. For any third random variable $X'$, jointly defined with $X$ and $Y$ on $\Omega$,
we prove that
\begin{align}\label{eq:shifted_chain_rule}
    \KL(\bs\mu^Y \mmid \bs\nu^Y)
    &\le \KL(\bs\mu^{X',Y} \mmid \bs \nu^{X,Y})
    \le \KL(\bs\mu^{X'} \mmid \bs \nu^X) + \int \KL(\bs \mu^{Y\mid X=x} \mmid \bs \nu^{Y\mid X=x'})\,\gamma(\D x, \D x')\,,
\end{align}
where $\gamma$ is any coupling of $\bs \mu^X$ and $\bs \mu^{X'}$.
Clearly,~\eqref{eq:shifted_chain_rule} contains~\eqref{eq:kl_chain_rule} as a special case (take $X = X'$), but the additional flexibility of introducing the auxiliary random variable $X'$ turns~\eqref{eq:shifted_chain_rule} into a powerful tool applicable to many situations where~\eqref{eq:kl_chain_rule} alone would not suffice.
Briefly, we modify the ``history'' of the process from $X\to Y$ to $X' \to Y$, at a price encapsulated in the second term on the right-hand side of~\eqref{eq:shifted_chain_rule}.
We refer to~\eqref{eq:shifted_chain_rule} (and its generalization to other divergences) as the \emph{shifted composition rule}. See Theorem~\ref{thm:scr} for the formal statement.

This series of papers investigates the shifted composition rule and its applications. In this first work, we focus on the application of this principle to deriving sharp \emph{Harnack inequalities} and \emph{reverse transport inequalities}.
To describe these results, we first provide some context.

To fix ideas, let $V : \R^d\to\R$ be a smooth function and consider the Langevin diffusion with potential $V$, namely, the solution to the It\^o stochastic differential equation (SDE)
\begin{align}\label{eq:langevin}
    \D X_t = -\nabla V(X_t) \, \D t + \sqrt 2 \, \D B_t\,,
\end{align}
where ${(B_t)}_{t\ge 0}$ is a standard Brownian motion on $\R^d$.
To study diffusion processes such as~\eqref{eq:langevin}, one usually introduces the corresponding Markov semigroup ${(P_t)}_{t\ge 0}$, which maps any (bounded) function $f : \R^d\to\R$ to $P_t f$ defined by $P_t f(x) \deq \E[f(X_t) \mid X_0 = x]$.
The analytic properties of the diffusion~\eqref{eq:langevin} (e.g., its regularizing effect) are then encoded as inequalities for the semigroup.
We refer to the monograph~\cite{bakry2014analysis} for a comprehensive account.

We will be particularly interested in the dimension-free Harnack inequality introduced in~\cite{Wang1997LSINoncompact}. In the context of~\eqref{eq:langevin}, this result reads as follows: suppose that $\nabla^2 V \succeq \alpha I$ on $\R^d$, for some $\alpha\in\R$; then, for any bounded non-negative function $f :\R^d\to\R$, and any $p > 1$,
\begin{align}\label{eq:p_harnack}
    {(P_t f(x))}^p
    &\le P_t(f^p)(y) \exp\biggl( \frac{\alpha p\,\norm{x-y}^2}{2\,(p-1)\,(\exp(2\alpha t)-1)}\biggr)\,, \qquad \forall\,x,y\in\R^d,\; t >0\,.
\end{align}
By replacing the Euclidean metric with an intrinsic metric,~\eqref{eq:p_harnack} holds more generally for Markov diffusions on Riemannian manifolds which satisfy the \emph{curvature-dimension condition} $\CD(\alpha,\infty)$, which reduces to $\nabla^2 V \succeq \alpha I$ for~\eqref{eq:langevin}.
In fact, as observed in~\cite{Wang10HarnackBoundary},~\eqref{eq:p_harnack} is \emph{equivalent} to $\CD(\alpha,\infty)$, and moreover to the reverse transport inequality
\begin{align}\label{eq:rev_transport}
    \Ren_q(\delta_x P_t \mmid \delta_y P_t)
    &\le \frac{\alpha q\,\norm{x-y}^2}{2\,(\exp(2\alpha t)-1)}\,, \qquad \forall\, x,y\in\R^d,\; t > 0\,,
\end{align}
where $q \deq \frac{p}{p-1}$ is the H\"older conjugate to $p$ and $\Ren_q$ is the R\'enyi divergence of order $q$ (see \S\ref{sec:prelim}).
We defer a more thorough discussion of the literature, including these equivalences, to \S\ref{ssec:apps:background}.

In~\cite{Wang1997LSINoncompact}, the Harnack inequality~\eqref{eq:p_harnack} was established using semigroup calculations based on the $\CD(\alpha,\infty)$ condition (or more precisely, based on certain gradient commutation bounds which are equivalent to $\CD(\alpha,\infty)$).
Then, in~\cite{ArnThaWan06HarnackCurvUnbdd}, M.\ Arnaudon, A.\ Thalmaier, and F.-Y.\ Wang introduced a coupling argument, which together with the Girsanov transformation, provides an alternative means of establishing inequalities such as~\eqref{eq:p_harnack}.
The latter approach has been used to systematically study SDEs on Riemannian manifolds, SDEs with multiplicative noise, SDEs with irregular coefficients, distribution-dependent SDEs, SPDEs, jump processes, and SDEs driven by fractional Brownian motion; we give citations to this extensive literature and revisit the coupling approach in \S\ref{sec:cont:sync}.

The formulation~\eqref{eq:rev_transport}, however, is formulated purely in terms of information-theoretic quantities, which naturally raises the question of obtaining a proof of~\eqref{eq:rev_transport} by means of an information-theoretic principle.
This is the starting point which motivates the present work.
Indeed, as we show in \S\ref{sec:disc}, the shifted composition rule can be used to recover~\eqref{eq:p_harnack} and~\eqref{eq:rev_transport} via elementary discrete-time arguments.
Moreover, through the information-theoretic lens, we unify, clarify, and refine concepts from distinct fields (namely, the shifted divergence technique from differential privacy~\cite{pabi} and the coupling argument of~\cite{ArnThaWan06HarnackCurvUnbdd}) and obtain new Harnack inequalities.
We now summarize the main contributions of our work.

\paragraph*{Contributions and organization.}
In \S\ref{sec:prelim}, we begin by reviewing the information-theoretic concepts that we employ, as well as their key properties.

In \S\ref{sec:disc}, we develop the discrete-time arguments which form the core technical innovation of our work.
We begin in \S\ref{ssec:disc:scp} by formally stating and proving the shifted composition rule (Theorem~\ref{thm:scr}).
Then, in \S\ref{ssec:disc:reduction}, we apply the shifted composition rule to prove sharp R\'enyi reverse transport inequalities (a.k.a.\ R\'enyi regularity bounds), of the form \begin{align}
    \Ren_q(\delta_x P^N \mmid \delta_y P^N) \le C\,\norm{x-y}^2
    \label{eq-intro:reg-dirac}
\end{align}
for discrete-time Markov kernels $P$ on $\R^d$ under the following two assumptions: (1) $P$ satisfies a \emph{one-step} regularity bound $\Ren_q(\delta_x P \mmid \delta_y P) \le c\,\norm{x-y}^2$, and (2) $P$ is Lipschitz in the $W_\infty$ metric.
Note that for our applications of interest, in which $P$ is taken to be a time-discretization of an SDE\@, the one-step regularity bound is typically easy to check since $P$ admits an explicit, Gaussian transition density.
For our result, given as Theorem~\ref{thm:discrete-main}, it is critical that we prove \emph{sharp} bounds in order to obtain non-trivial results for the continuous-time diffusion as we let the discretization time step tend to zero, as well as to recover the aforementioned sharp equivalences with the $\CD(\alpha,\infty)$ condition (see Remark~\ref{rem:multi-step}).

Our sharp regularity bounds for discrete-time Markov processes are obtained through the introduction of auxiliary ``\emph{shifted}'' processes and appealing to the shifted composition rule.
In fact, we identify two separate auxiliary processes which suffice for this purpose: one based on \emph{synchronous} coupling, and one based on coupling via \emph{optimal mass transport}.
In turn, they lead to two different continuous-time arguments based, respectively, on stochastic calculus and optimal transport.

Next, in \S\ref{ssec:disc:convexity}, we show that regularity bounds~\eqref{eq-intro:reg-dirac} which hold for Dirac initializations $\delta_x$, $\delta_y$ can be upgraded, in a black-box manner, to regularity bounds of the form
\begin{align*}
    \Ren_q(\mu P^N \mmid \nu P^N) \le C\, W^2(\mu,\nu)
\end{align*}
that hold from arbitrary initializations $\mu$, $\nu$, and replace the the quadratic cost $\norm{x-y}^2$ on the right-hand side of the regularity bound by a suitable \emph{coupling cost} $W^2(\mu,\nu)$ between $\mu$ and $\nu$ (Theorem~\ref{thm:renyi_cvxty_principle}).
Although the argument is straightforward, based on the joint convexity of information divergences and a coupling argument, we have not found a self-contained statement of this principle in the literature.
Taken together with the result of \S\ref{ssec:disc:reduction}, this yields a general reduction in which, for $W_\infty$-Lipschitz kernels $P$, sharp multi-step regularity bounds from general initializations follow from one-step regularity bounds from Dirac initializations.

In \S\ref{ssec:disc:langevin} we illustrate our results for the Langevin SDE~\eqref{eq:langevin}. In this context, our arguments closely resemble the \emph{shifted divergence} technique from the field of differential privacy~\cite{pabi}, and in particular the modified version~\cite{AltTal22dp} which was recently used to established discrete mixing bounds~\cite{AltChe23warm, AltTal23Langevin}.
We discuss these connections further in \S\ref{ssec:discrete:dp}.
In brief, our framework can be viewed as a generalization, refinement, and interpretation of the shifted divergence method.
It is a generalization since we have identified the fundamental information-theoretic principle---namely, the shifted composition rule---underlying the method, which allows for more general notions of shifts than Gaussian convolutions; it is a refinement due to the convexity principle of \S\ref{ssec:disc:convexity}, which improves both quantitatively and qualitatively over prior results in the literature (e.g.,~\cite{AltChe23warm,AltTal23Langevin}); and it provides meaningful interpretations through the construction of \emph{explicit} shifted processes, as well as by connecting it to the corresponding continuous-time arguments in \S\ref{sec:cont}.
The freedom to choose more general ``shifts'' will be exploited in further works in this series.

Next,~\S\ref{sec:cont} develops the corresponding arguments in continuous time; in particular, the two choices for the auxiliary process lead to conceptually distinct proofs.
In \S\ref{sec:cont:sync} we show that the continuous-time analogue of the synchronous coupling proof coincides with the aforementioned ``coupling by parallel translation'' introduced by~\cite{ArnThaWan06HarnackCurvUnbdd}. Hence the arguments of \S\ref{sec:disc} can be viewed as a way to extend the method of~\cite{ArnThaWan06HarnackCurvUnbdd}, based on Girsanov's Theorem, to discrete-time Markov processes. On the other hand, the continuous-time analogue of the Wasserstein coupling proof, given in \S\ref{sec:cont:opt}, relies on calculations in the spirit of Otto calculus~\cite{otto01porousmedium} (c.f.~\cite{jordan1998variational, AGS, villani2009optimal}) and appears to be new. This argument bypasses the need for Girsanov transformations.
We also discuss links with the F\"ollmer process~\cite{Fol1985Reversal} and the ``JKO'' (or minimizing movements) scheme, which may be conceptually useful.

In \S\ref{sec:ext}, we explore further extensions of our results, starting with a word on the Riemannian setting in \S\ref{ssec:ext:riem} and proceeding to general It\^o SDEs on $\R^d$ with multiplicative noise in \S\ref{ssec:ext:ito}.
The latter setting was first considered by F.-Y.\ Wang in~\cite{Wang11HarnackMultNoise}, and we show how to recover his sharp log-Harnack inequality via discrete-time arguments.
Then, in \S\ref{ssec:ext:clt}, we consider the Markov kernel induced by $N$-fold convolution with a regular density $\rho$ under the central limit scaling and we obtain a regularity bound depending on the Fisher information matrix for $\rho$.
We conjecture that the Fisher information matrix can be replaced by the inverse covariance matrix, which would be sharp.
We stress that continuous-time coupling arguments do not apply to the study of these discrete-time processes.

Finally, in \S\ref{sec:apps}, we discuss applications of our results to the study of Harnack inequalities.
We start with background in \S\ref{ssec:apps:background} on the equivalence between $\CD(\alpha,\infty)$, Harnack inequalities, and reverse transport inequalities; in particular, we emphasize the duality between the latter two in \S\ref{ssec:apps:duality}.
Hence, our regularity bounds/reverse transport inequalities immediately furnish Harnack inequalities, in particular for discrete-time processes (see \S\ref{ssec:apps:discrete}).

In \S\ref{ssec:apps:reverse}, we show that dualizing our regularity bounds for R\'enyi parameters $q \in (0,1)$ yields a family of \emph{reverse} Harnack inequalities which correspond to exponents $p \in (-\infty, 0)$.
These inequalities have not previously appeared in the literature, and we prove that they are also equivalent to the curvature-dimension condition $\CD(\alpha,\infty)$.

For brevity, we defer some calculations and proofs to the appendices \S\ref{app:deferred} and \S\ref{app:dual}.
    \section{Information-theoretic preliminaries}\label{sec:prelim}

Here we briefly recall the definitions and basic properties of the information divergences employed in this paper.

\begin{defin}[R\'enyi divergence]
    Let $q \in (0,\infty]$. The \emph{R\'enyi divergence} of order $q$ between probability measures $\mu$, $\nu$ is defined to be
    \begin{align}\label{eq:renyi}
        \Ren_q(\mu \mmid \nu)
        &\deq \frac{1}{q-1} \log \int \bigl(\frac{\D\mu}{\D\nu}\bigr)^q \, \D \nu\,.
    \end{align}
    For $q = 1$, this is known as the \emph{Kullback{--}Leibler (KL) divergence} and we interpret~\eqref{eq:renyi} in the limiting sense,
    \begin{align*}
        \KL(\mu \mmid \nu)
        \deq \Ren_1(\mu \mmid \nu)
        \deq \int \Bigl(\frac{\D \mu}{\D\nu} \log \frac{\D\mu}{\D\nu}\Bigr) \, \D \nu\,.
    \end{align*}
    For $q = \infty$, we again interpret~\eqref{eq:renyi} in the limiting sense,
    \begin{align*}
        \Ren_\infty(\mu\mmid \nu)
        &\deq \log {\Bigl\lVert \frac{\D\mu}{\D\nu}\Bigr\rVert_{L^\infty(\nu)}}\,.
    \end{align*}
    If $\mu \not\ll \nu$, then $\Ren_q(\mu \mmid \nu)$ is defined to be $+\infty$ for $q > 1$, and $\Ren_q(\mu\mmid\nu) \deq \frac{1}{q-1} \log \int (\frac{\D\mu}{\D\lambda})^q \, (\frac{\D\nu}{\D\lambda})^{1-q} \, \D \lambda$ for $q < 1$, where $\lambda$ is a common dominating measure for $\mu$ and $\nu$ (e.g., $\lambda = \mu+\nu$).
\end{defin}

Another special case worth remarking is $q=2$, in which case the R\'enyi divergence is related to the \emph{chi-squared divergence}
\begin{align*}
    \chi^2(\mu\mmid \nu)
    &\deq \var_\nu \frac{\D\mu}{\D\nu}
    = \int \Bigl(\frac{\D\mu}{\D\nu}\Bigr)^2 \, \D \nu - 1\,,
\end{align*}
via the expression $\Ren_2(\mu\mmid \nu) = \exp(1+\chi^2(\mu \mmid \nu))$.

For later convenience, we define $\Hell_q$ for $q \ne 1$ to be the $f$-divergence corresponding to
\begin{align*}
    f_q(x) \deq \begin{cases}
        x^q - 1\,, & q > 1\,, \\
        1-x^q\,, & q < 1\,.
    \end{cases}
\end{align*}
In other words, we set
\begin{align*}
    \Hell_q(\mu \mmid \nu)
    &\deq \int f_q\bigl( \frac{\D\mu}{\D\nu}\bigr) \, \D \nu\,.
\end{align*}
Note that $f_q$ is convex with $f_q(1) = 0$. We have the relationships
\begin{align}\label{eq:ren_and_hell}
    \Ren_q(\mu \mmid \nu) = \frac{1}{q-1} \begin{cases}
        \log(1+\Hell_q(\mu\mmid \nu))\,, & q > 1\,, \\
        \log(1-\Hell_q(\mu\mmid \nu))\,, & q < 1\,.
    \end{cases}
\end{align}

We also summarize a number of standard properties of R\'enyi divergences that we use repeatedly throughout the paper. Proofs and further discussion of these properties can be found, e.g., in the surveys~\cite{van2014renyi,mironov2017renyi}. Since we provide a slightly modified restatement of the R\'enyi composition rule that is helpful for our development, we provide a proof in \S\ref{app:pf_composition} for completeness.

\begin{theorem}\label{thm:renyi_prop}
    Let $q \in (0,\infty]$ and let $\mu$, $\nu$ be probability measures.
    \begin{enumerate}
        \item (Positivity) $\Ren_q(\mu \mmid \nu) \ge 0$, with equality if and only if $\mu = \nu$.
        \item (Monotonicity) R\'enyi divergences are increasing in the order, i.e., $q \mapsto \Ren_q(\mu \mmid \nu)$ is increasing.
        \item (Data processing inequality) For any Markov kernel $P$, it holds that $\Ren_q(\mu P \mmid \nu P) \le \Ren_q(\mu \mmid \nu)$.
        \item (KL chain rule) Using the notation introduced in \S\ref{sec:intro},
        \begin{align*}
            \KL(\bs\mu^{X,Y} \mmid \bs \nu^{X,Y})
            &= \KL(\bs \mu^X \mmid \bs \nu^X) + \int \KL(\bs \mu^{Y\mid X=x} \mmid \bs \nu^{Y\mid X=x}) \, \bs \mu^X(\D x)\,.
        \end{align*}
        \item (R\'enyi composition rule) For $q \in (0, 1)$,
        \begin{align*}
            \Ren_q(\bs\mu^{X,Y} \mmid \bs \nu^{X,Y})
            &\le \Ren_q(\bs \mu^X \mmid \bs \nu^X) + \esssup{(\bs\mu^X\wedge \bs\nu^X)}{\bigl[\Ren_q(\bs \mu^{Y\mid X=\bullet} \mmid \bs \nu^{Y\mid X=\bullet})\bigr]}\,.
        \end{align*}
        For $q\ge 1$,
        \begin{align}\label{eq:composition}
            \Ren_q(\bs\mu^{X,Y} \mmid \bs \nu^{X,Y})
            &\le \Ren_q(\bs \mu^X \mmid \bs \nu^X) + \esssup{\bs\mu^X}{\bigl[ \Ren_q(\bs \mu^{Y\mid X=\bullet} \mmid \bs \nu^{Y\mid X=\bullet})\bigr]}\,.
        \end{align}
        \item (Convexity) The divergences $\Hell_q$ (for $q\ne 1$) and $\Ren_q$ (for $q \le 1$) are jointly convex.
        Consequently, since $\Ren_q$ is an increasing transformation of $\Hell_q$ for $q > 1$, it follows that $\Ren_q$ is jointly quasi-convex for the entire range $q > 0$.
        \item (Gaussian identity) $\Ren_q(\cN(x,\sig^2 I) \mmid \cN(y,\sig^2 I)) = \frac{q\,\|x-y\|^2}{2\sig^2}$. 
    \end{enumerate}
\end{theorem}

\begin{remark}
    In the composition rule~\eqref{eq:composition}, it is important that the essential supremum on the RHS is taken w.r.t.\ $\bs \mu^X$.
    Indeed, for~\eqref{eq:composition}, we may assume that $\bs\mu^X \ll \bs \nu^X$ or else the bound is trivial.
    The conditional distribution $\bs \nu^{Y\mid X=\bullet}$ is defined $\bs\nu^X$-a.e., hence $\bs \mu^X$-a.e., and the expression on the RHS of~\eqref{eq:composition} therefore makes sense.
    On the other hand, $\bs \mu^{Y\mid X=\bullet}$ may not be defined $\bs \nu^X$-a.e.
\end{remark}

The composition rule is key to our work as it enables proving the shifted composition rule in \S\ref{ssec:disc:scp}. As such, we focus on the family of R\'enyi divergences rather than other $f$-divergences.
    % !TEX root = ../scp1.tex

\section{Discrete-time arguments}\label{sec:disc}

\subsection{Shifted composition rule}\label{ssec:disc:scp}

The namesake of this paper (and the forthcoming series) is the following \emph{shifted composition rule}. Write $\Coup(\mu,\nu)$ for the set of couplings of two probability measures $\mu \in \mc P(\Omega_1)$, $\nu \in \mc P(\Omega_2)$, i.e., the set of probability measures $\gamma \in \mc P(\Omega_1\times \Omega_2)$ whose marginals are $\mu$ and $\nu$ respectively.

\begin{theorem}[Shifted composition rule]\label{thm:scr}
    Let $X$, $X'$, $Y$ be three jointly defined random variables on a standard probability space $\Omega$.
    Let $\bs\mu$, $\bs \nu$ be two probability measures over $\Omega$, with superscripts denoting the laws of random variables under these measures.
    \begin{enumerate}
        \item (Shifted chain rule) It holds that
        \begin{align*}
            \KL(\bs\mu^Y \mmid \bs \nu^Y)
            &\le \KL(\bs\mu^{X'} \mmid \bs \nu^X) + \inf_{\gamma \in \Coup(\bs \mu^X, \bs \mu^{X'})} \int \KL(\bs\mu^{Y\mid X=x} \mmid \bs \nu^{Y\mid X=x'}) \,\gamma(\D x, \D x')\,.
        \end{align*}
        \item Let $q \in (0,\infty]$. If $q \in (0,1)$, assume in addition that $\bs\mu^{X'} \ll \bs\nu^X$. Then, it holds that
        \begin{align*}
            \Ren_q(\bs\mu^Y \mmid \bs \nu^Y)
            &\le \Ren_q(\bs\mu^{X'} \mmid \bs \nu^X) + \inf_{\gamma \in \Coup(\bs \mu^X, \bs \mu^{X'})} {\esssup{\gamma}_{(x,x') \in \Omega \times \Omega}{\Ren_q(\bs\mu^{Y\mid X=x} \mmid \bs \nu^{Y\mid X=x'})}}\,.
        \end{align*}
    \end{enumerate}
\end{theorem}
\begin{proof}
    First, note that statement we wish to prove only depends on the laws $\bs \mu^{X,Y}$, $\bs \nu^{X,Y}$, and $\bs \mu^{X'}$, and hence we are free to choose the coupling between $X'$ and $(X,Y)$.
    Given any coupling $\gamma \in \Coup(\bs\mu^X, \bs \mu^{X'})$, we can jointly define $(X,X',Y)$ such that $\bs\mu^{X,X'} = \gamma$ using the \emph{gluing lemma} (see~\cite[Lemma 7.6]{Vil03Topics}), i.e., we set $\bs \mu^{X,X',Y}(\D x, \D x', \D y) = \bs \mu^{X,Y}(\D x, \D y) \, \gamma^{2\mid 1}(\D x' \mid x)$, where $\gamma^{2\mid 1}$ denotes the disintegration of $\gamma$ along the first coordinate.
    Note that with this choice, under $\bs \mu$, $X'$ and $Y$ are conditionally independent given $X$, i.e., $X'\to X \to Y$ form a $\bs\mu$-Markov chain.
    
    By the data processing inequality and the KL chain rule or the R\'enyi composition rule respectively, and using $\bs\mu^{X'} \ll \bs \nu^X$ to write $\esssup{(\bs\mu^{X'} \wedge\bs\nu^X)} = \esssup{\bs\mu^{X'}}$,
    \begin{align*}
        \KL(\bs\mu^Y \mmid \bs \nu^Y)
        &\le \KL(\bs\mu^{X',Y} \mmid \bs \nu^{X,Y})
        = \KL(\bs\mu^{X'} \mmid \bs \nu^X) + \int \KL(\bs \mu^{Y\mid X'=x'} \mmid \bs \nu^{Y\mid X=x'}) \, \bs \mu^{X'}(\D x')\,, \\
        \Ren_q(\bs\mu^Y \mmid \bs \nu^Y)
        &\le \Ren_q(\bs\mu^{X',Y} \mmid \bs \nu^{X,Y})
        \le \Ren_q(\bs\mu^{X'} \mmid \bs \nu^X) + \esssup{\bs\mu^{X'}}_{x'\in \Omega} \Ren_q(\bs \mu^{Y\mid X'=x'} \mmid \bs \nu^{Y\mid X=x'})\,.
    \end{align*}
    Next, by conditioning, we write
    \begin{align*}
        \bs \mu^{Y\mid X'=x'} = \int \bs\mu^{Y\mid X=x, X'=x'} \, \bs \mu^{X\mid X'}(\D x \mid x')
        = \int \bs\mu^{Y\mid X=x} \, \bs \mu^{X\mid X'}(\D x \mid x')
    \end{align*}
    where we used the fact that under $\bs \mu$, $X'$ and $Y$ are conditionally independent given $X$.
    Using the convexity of the KL divergence and the quasi-convexity of the R\'enyi divergence,
    \begin{align*}
        \int \KL(\bs \mu^{Y\mid X'=x'} \mmid \bs \nu^{Y\mid X=x'}) \, \bs \mu^{X'}(\D x')
        &\le \int \KL(\bs \mu^{Y\mid X=x} \mmid \bs \nu^{Y\mid X=x'}) \, \bs \mu^{X\mid X'}(\D x \mid x')\,\bs \mu^{X'}(\D x') \\
        &= \int \KL(\bs \mu^{Y\mid X=x} \mmid \bs \nu^{Y\mid X=x'}) \, \bs \mu^{X, X'}(\D x,\D x')\,.
    \end{align*}
    and
    \begin{align*}
        \esssup{\bs\mu^{X'}}_{x'\in \Omega} \Ren_q(\bs \mu^{Y\mid X'=x'} \mmid \bs \nu^{Y\mid X=x'})
        &\le \esssup{\bs\mu^{X,X'}}_{(x,x')\in \Omega\times \Omega} \Ren_q(\bs \mu^{Y\mid X=x} \mmid \bs \nu^{Y\mid X=x'})\,.
    \end{align*}
    The conclusion follows because $\bs\mu^{X,X'} = \gamma \in \Coup(\bs\mu^X,\bs\mu^{X'})$ was arbitrary.
\end{proof}

For $q \ge 1$, if we take $\bs\mu^X = \bs \mu^{X'}$ and we take $\gamma$ to be the trivial coupling $\gamma(\D x, \D x') = \bs\mu^X(\D x) \, \delta_x(\D x')$, then the shifted composition rule reduces back to the KL chain rule or the R\'enyi composition rule respectively.
However, the added flexibility of introducing the auxiliary random variable $X'$ allows the shifted composition rule to tackle a variety of new applications, some of which will be explored in future work.
In this paper, we illustrate the use of this principle for proving Harnack and reverse transport inequalities, as discussed in \S\ref{sec:intro}.

\subsection{One-step to multi-step bounds}\label{ssec:disc:reduction}

We now turn toward the main application of the shifted composition rule considered in the present paper, namely, the derivation of reverse transport inequalities of the form $\Ren_q(\delta_x P^n \mmid \delta_y P^n) \lesssim \norm{x-y}^2$, where $P$ is a Markov kernel on $\R^d$ and $0 < q \le \infty$. We also refer to such inequalities as \emph{regularity bounds} since they encode regularizing properties of the Markov kernel $P$, see \S\ref{ssec:apps:background} for further discussion. Our result below shows that if the Markov kernel $P$ is Lipschitz w.r.t.\ the Wasserstein metric, then an optimal \emph{multi-step} regularity bound is implied by a \emph{one-step} regularity bound $\Ren_q(\delta_x P \mmid \delta_y P) \lesssim \|x-y\|^2$, which is typically much easier to establish.
In the next section, we will then show how to upgrade the regularity bounds to hold for arbitrary initializations $\mu,\nu \in \mc P(\R^d)$ in a black-box manner.

\begin{theorem}\label{thm:discrete-main}
    Let $0 < q \le \infty$.
    Suppose that $P$ is a Markov kernel on $\R^d$ satisfying the two following conditions.
	\begin{enumerate}[label=(\alph*)]
		\item\label{ass:one_step_reg} $P$ satisfies a $1$-step regularity bound for Dirac initializations; i.e., there exists $c > 0$ such that
		\begin{align}\label{eq:one_step_reg}
			\Ren_q(\delta_x P \mmid \delta_y P) \leq c\,\|x-y\|^2\, \qquad \forall x,y \in \R^d\,.
		\end{align}
          If $q < 1$, we assume for technical reasons that~\eqref{eq:one_step_reg} also holds for $q=1$, possibly with some other constant $c' < \infty$.
		\item\label{ass:wass_lip} $P$ is Wasserstein-Lipschitz; i.e., there exists $L > 0$ such that
		\begin{align*}
			W_{\infty}(\mu P, \nu P) \leq L\, W_{\infty}(\mu,\nu)\,, \qquad \forall \mu,\nu \in \cP(\R^d)\,.
		\end{align*}
	\end{enumerate}
	Then, for all $x,y\in\R^d$,
	\begin{align}
	\Ren_q(\delta_x P^N \mmid \delta_y P^N) &\leq  c \, \frac{L^{-2}-1}{L^{-2N}-1} \, \norm{x-y}^2\,.
 \label{eq-thm:discrete-main}
	\end{align}
\end{theorem}

\begin{remark}[Optimality of the bound]\label{rem:multi-step}
    The multi-step bound~\eqref{eq-thm:discrete-main} is optimal in the absence of further assumptions on $P$ (see \S\ref{ssec:disc:langevin}). We remark that while it is trivial to prove the weaker bound $\Ren_q(\delta_x P^N \mmid \delta_y P^N) \le cL^{2N-2}\,\norm{x-y}^2$ by applying Assumption~\ref{ass:one_step_reg} for one step and~\ref{ass:wass_lip} for the remaining steps, that na\"{\i}ve bound is weaker to the point of being vacuous\footnote{E.g., for the Langevin SDE discussed in \S\ref{ssec:disc:langevin}, this na\"{\i}ve argument gives a R\'enyi regularity bound of order $O(h^{-1})$ which is vacuous in the continuous time limit $h \searrow 0$. This bounds is vacuous because it only makes use of a vanishing amount of regularization (just one step).}
    for the applications we have in mind, namely time discretizations of diffusions with small step size parameter $h > 0$. 
\end{remark}

\begin{remark}[$W_{\infty}$-Lipschitz assumption]\label{rmk:improving_winf}
    For simplicity, we state Theorem~\ref{thm:discrete-main} under the assumption that $P$ is Lipschitz in the $W_{\infty}$ distance. This enables covering all of the R\'enyi divergences using the same proof.
    However, for the KL divergence ($q=1$), the $W_{\infty}$-Lipschitz assumption can be relaxed to a $W_2$-Lipschitz assumption, which is strictly weaker. This follows by replacing occurrences of $W_{\infty}$ with $W_2$ in the relevant parts of the proof, namely in~\eqref{eq:pf-thm-discrete:distance} and~\eqref{eq:pf-thm-discrete:div} (using the shifted chain rule specific to the KL divergence). 
\end{remark}

\begin{remark}[Verifying Wasserstein-Lipschitzness]
    In order to verify the assumption that $P$ is $W_2$- or $W_\infty$-Lipschitz, it suffices to check this condition when $\mu,\nu$ are Dirac measures. This follows from an elementary coupling argument, see e.g.,~\cite[\S A.2]{chenetal2022proximalsampler} for the proof in the $W_2$ case.

\end{remark}

\begin{figure}
\caption{\footnotesize Both the synchronous coupling and Wasserstein coupling approaches produce an auxiliary stochastic process $\{\mu_n'\}_{n=0}^N$ that interpolates between $\{\mu_n\}_{n=0}^N$ and $\{\nu_n\}_{n=0}^N$ in the sense that $\mu_0' = \nu$ and $\nu_N' = \nu_N$.}
\centering
\includegraphics[width=0.4\textwidth]{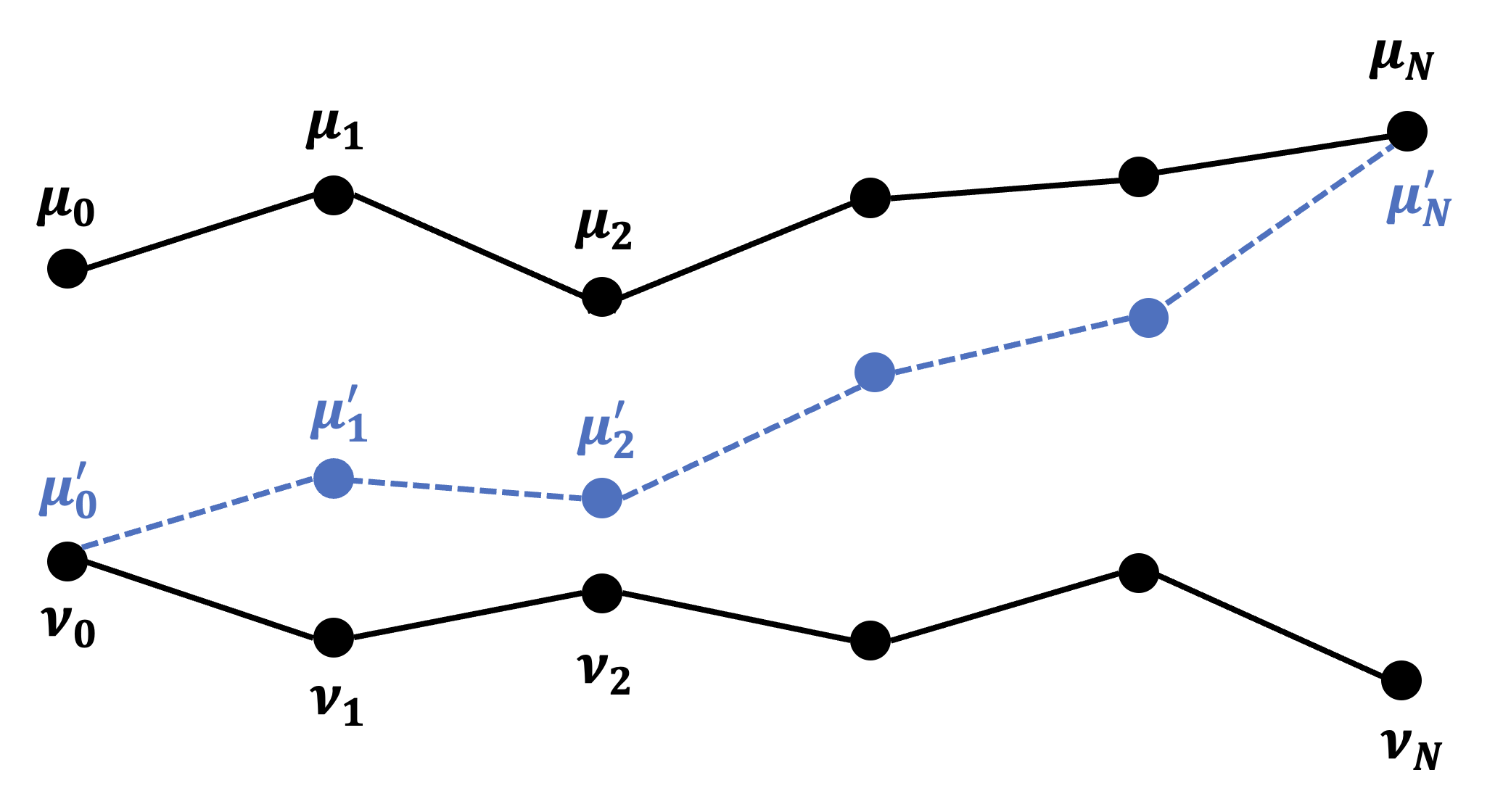}
\label{fig:aux-process}
\end{figure}

\par 
The key ingredient in our proof is the introduction of a \emph{shifted interpolated process}. Since this is a central element of our analysis---for the discrete-time arguments in this section as well as for the continuous-time arguments in \S\ref{sec:cont}---we isolate this idea before the proof. We construct an auxiliary process $\{\mu_n'\}_{n=0}^N$ that interpolates between the processes $\{\mu_n \deq \delta_x P^n\}_{n=0}^N$ and $\{\nu_n  \deq  \delta_y P^n\}_{n=0}^N$, in the sense that it matches one at initialization and matches the other at termination:
 \begin{align}
     \mu_0' = \nu_0 \qquad \text{and} \qquad \mu_N' = \mu_N\,.
     \label{eq:interp-disc}
 \end{align}
 See Figure~\ref{fig:aux-process}.
Since $\mu_N' = \mu_N$, it obviously holds that the left hand side of the regularity bound~\eqref{eq-thm:discrete-main} is equal to $\Ren_q(\mu_N' \mmid \nu_N) = \Ren_q(\mu_N \mmid \nu_N)$.
The insight behind this construction is that rather than using the composition rule to bound the divergence between the original processes $\{\mu_n\}_{n=0}^N$ and $\{\nu_n\}_{n=0}^N$, it is more efficient to bound the divergence between the \emph{auxiliary} process $\{\mu_n'\}_{n=0}^N$ and $\{\nu_n\}_{n=0}^N$ using the \emph{shifted} composition rule. (In fact, the divergence between the original processes is infinite since $\mu_0 = \delta_x$ and $\nu_0 = \delta_y$ are singular w.r.t.\ each other.)
 
\par There are two constructions of $\{\mu_n'\}_{n=0}^N$ that suffice for our purpose. For both, we set 
\begin{align*}
 \mu_n'  \deq \law(X_n')\,, \qquad n=0,1,\dotsc,N\,,
\end{align*}
for a stochastic process $\{X_n'\}_{n=0}^N$ to be defined below. The interest in considering these two shifted interpolated processes is that they lead to different generalizations in continuous time, as demonstrated in \S\ref{sec:cont}.

\paragraph*{Synchronous coupling.}
Jointly define processes $\{X_n\}_{n=0}^N$, $\{X_n'\}_{n=0}^N$ such that $X_n\sim\mu_n$ and $X_n' \sim \mu_n'$ for all $n$.
We start with $X_0 = x$ and $X_0' = y$.
Assuming that $(X_n, X_n')$ have been jointly defined, define $(X_{n+1}, X_{n+1}')$ as follows. Set
\begin{align}\label{eq:sync_coup}
    \tilde X_n
    &\deq X_n' + \eta_n \,(X_n - X_n')
\end{align}
for a scalar $\eta_n \ge 0$ to be chosen later.
Then, conditional on $(X_n, X_n')$, draw $X_{n+1} \sim P(X_n,\cdot)$, $X_{n+1}'\sim P(\tilde X_n,\cdot)$ so that $\norm{X_{n+1}-X_{n+1}'}_{L^\infty(\Pr)} = W_\infty(P(X_n,\cdot), P(\tilde X_n,\cdot))$.

\paragraph*{Wasserstein coupling.} Assuming that $X_n'$ has already been defined, let $X_n\sim \mu_n$ be \emph{optimally} coupled with $X_n'$ for the $W_\infty$ metric and define $\tilde X_n$ via~\eqref{eq:sync_coup}.
Then, conditional on $\tilde X_n$, draw $X_{n+1}'\sim P(\tilde X_n,\cdot)$.

\paragraph*{Remarks on the constructions.}
    In both constructions, $\eta_n$ controls how much the process $X_n'$ is corrected in the direction of $X_n$. This enables us to make progress in each iteration towards achieving the termination criterion of matching $\mu_N' = \law(X_N')$ to $\mu_N = \law(X_N)$. 
    In both settings, we take $\eta_{N-1} = 1$ so that $\mu_N' = \mu_N$, and we optimize the other shifting parameters $\eta_0,\dots, \eta_{N-2}$ below to obtain the best possible final bound. 

    \par In order to prove a bound of the form~\eqref{eq-thm:discrete-main}, we make two key observations. First, the distance between $\mu_n$ and $\mu_n'$ contracts in each iteration. Second, the divergence $\Ren_q(\mu_{n}' \mmid \nu_{n})$ can be controlled via the shifted composition rule.
    
\begin{proof}[Proof of Theorem~\ref{thm:discrete-main}]
\textbf{Distance bound for the auxiliary process.}
We give the argument for the synchronous shifted interpolation.
Below, we work over an underlying probability space $(\Omega,\ms F, \Pr)$.
Via a coupling argument, almost surely,
\begin{align*}
    \norm{X_{n+1}-X_{n+1}'}
    \le W_\infty(P(X_n,\cdot), P(\tilde X_n,\cdot))
    \le L\,\norm{X_n-\tilde X_n}
    = L\,\abs{1-\eta_n}\,\norm{X_n - X_n'}\,,
\end{align*}
hence
\begin{align}
    \norm{X_{n+1}-X_{n+1}'}_{L^\infty(\Pr)}
    &\le L\,\abs{1-\eta_n}\,\norm{X_n - X_n'}_{L^\infty(\Pr)}\,.\label{eq:pf-thm-discrete:distance}
\end{align}
	Above, the second inequality is by the Lipschitz assumption~\ref{ass:wass_lip} on the kernel $P$.
    By iterating this bound and recalling that $\mu_0' = \delta_y$ by construction, we conclude that for all $n$,
    \begin{align}\label{eq:main_thm_dist_bd}
        \norm{X_n-X_n'}_{L^\infty(\Pr)} \le \Bigl[L^{n} \prod_{k=0}^{n-1} \abs{1-\eta_k} \Bigr]\,\norm{x-y} \,.
    \end{align}
    The distance bound for the Wasserstein interpolated process is similar, except that in~\eqref{eq:main_thm_dist_bd} we replace the left-hand side with $W_\infty(\mu_n,\mu_n')$.
    
    \textbf{Divergence bound for the auxiliary process.} The second key bound controls $\Ren_q$ between the auxiliary process $\mu_n'$ and $\nu_n$. Again, we consider the synchronous shifted interpolation.
 \begin{align}
		\Ren_q(\mu_{n+1}' \mmid \nu_{n+1})
		&\leq
		\Ren_q(\mu_n' \mmid \nu_n ) + \bigl\lVert \Ren_q( P(\tilde{X}_n,\cdot) \mmid P(X_n',\cdot) )\bigr\rVert_{L^\infty(\Pr)} \nonumber
		\\ &\leq
		\Ren_q(\mu_n' \mmid \nu_n ) + c\, \norm{\tilde X_n - X_n'}_{L^\infty(\Pr)}^2 \nonumber
		\\ &= \Ren_q(\mu_n' \mmid \nu_n) + c \eta_n^2\, \norm{X_n-X_n'}_{L^\infty(\Pr)}^2 \,.
        \label{eq:pf-thm-discrete:div}
	\end{align} 
    Above, the first step is by an application of the shifted composition rule\footnote{When $q<1$, we must check that $\mu_n' \ll \nu_n$ for all $n=0,1,\dotsc,N$. However, this follows from our argument since we have assumed in this case that the assumption~\ref{ass:one_step_reg} also holds for $q=1$. Our proof therefore shows that $\KL(\mu_n'\mmid \nu_n) <\infty$.} (Theorem~\ref{thm:scr}) where $\bs{\mu}$ is the joint distribution under which $X \sim \tilde{\mu}_n$, $X' \sim \mu_n'$, and $Y \sim P(X,\cdot)$; and $\bs{\nu}$ is the joint distribution under which $X \sim \nu_n$ and $Y \sim P(X,\cdot)$. The second step is by the assumption~\ref{ass:one_step_reg}. The final step is by construction of $\tilde{X}_n$. 
    For the Wasserstein interpolated process, we replace the second term on the right-hand side of~\eqref{eq:pf-thm-discrete:div} with $c\eta_n^2\,W_\infty^2(\mu_n,\mu_n')$.
    
	\textbf{Optimizing the shifts.} 
    Combining the two bounds above yields
	\[	
 \Ren_q(\mu_N \mmid \nu_N) =
    \Ren_q(\mu_N' \mmid \nu_N) \leq 
    \Bigl[c \sum_{n=0}^{N-1} L^{2n} \eta_n^2 \prod_{k=0}^{n-1} (1 - \eta_k)^2\Bigr]\,\norm{x-y}^2\,.
	\]
    Recall that this bound holds for any values of the shifts $\eta_0, \dots, \eta_{N-1}$ subject to the constraint $\eta_{N-1}=1$ (required to ensure the termination criterion $\mu_N' = \mu_N$). Thus we may optimize the above bound over all such choices of $\eta$. This optimization problem is straightforward to solve in closed form, as detailed in \S\ref{app:calc-discrete}. Plugging in the optimal value $\frac{L^{-2}-1}{L^{-2N}-1}$ completes the proof.
\end{proof}

\par We remark in passing that our analysis readily generalizes to non-stationary processes in which different Markov kernels are applied in each iteration.

\subsection{Convexity principle}\label{ssec:disc:convexity}

In the previous subsection, we proved regularity bounds for Markov chains initialized at Dirac distributions. Here, we reduce the problem of proving regularity bounds from \emph{arbitrary} initializations to the case of Dirac initializations. The simple but key observation underlying this reduction is the following \emph{convexity principle}.

\begin{lemma}[Convexity principle]\label{lem:convexity}
    For any jointly convex function $\cD$, any Markov kernel $P$, and any distributions $\mu,\nu$,
    \begin{align}
        \cD(\mu P \mmid \nu P) 
        \leq
        \inf_{\gamma \in \Coup(\mu,\nu)} \int \cD(\delta_x P \mmid \delta_y P)\,\gamma(\D x,\D y)\,.
    \end{align}    
\end{lemma}
\begin{proof}
    Fix any coupling $\gamma \in \Coup(\mu,\nu)$. Decompose $\mu$ as the mixture distribution $\int \delta_x\, \gamma(\D x,\D y)$, and similarly decompose $\nu = \int \delta_y\, \gamma(\D x,\D y)$. Joint convexity then implies
    \begin{align*}
        \cD(\mu P \mmid \nu P)
        &= \cD\Bigl(\int \delta_x P\, \gamma(\D x,\D y) \Bigm\Vert \int \delta_y P\, \gamma(\D x,\D y)\Bigr)
        \leq \int \cD(\delta_x P \mmid \delta_y P) \,\gamma(\D x,\D y)\,.
    \end{align*}
    The claim follows since $\gamma$ is an arbitrary coupling.
\end{proof}

We apply the convexity principle to the family of R\'enyi divergences by exploiting the basic fact from information theory that $f$-divergences are jointly convex (Theorem~\ref{thm:renyi_prop}).

\begin{theorem}[Application to R\'enyi divergences]\label{thm:renyi_cvxty_principle}
    Let $P$ be a Markov kernel on a Polish space $\cX$ and let $\rho$ be a measurable function on $\cX\times \cX$.
    \begin{enumerate}
        \item If $\KL(\delta_x P \mmid \delta_y P) \le \rho(x,y)$ for all $x,y\in\cX$, then
        \begin{align}\label{eq:kl_cvxty_principle}
            \KL(\mu P \mmid \nu P)
            &\le \inf_{\gamma\in\Coup(\mu,\nu)} \int \rho(x,y) \, \gamma(\D x, \D y) \qquad\text{for all}~\mu,\nu\in\cP(\cX)\,.
        \end{align}
        \item Let $q \in (0,\infty) \setminus \{1\}$. If $\Ren_q(\delta_x P \mmid \delta_y P) \le \rho(x,y)$ for all $x,y\in\cX$, then
        \begin{align}\label{eq:ren_cvxty_principle}
            \Ren_q(\mu P \mmid \nu P)
            &\le \inf_{\gamma \in \Coup(\mu,\nu)} \frac{1}{q-1} \log \int \exp\bigl((q-1)\,\rho(x,y)\bigr)\,\gamma(\D x, \D y)\,.
        \end{align}
    \end{enumerate}
\end{theorem}
\begin{proof}
    Since the $\KL$ divergence is an $f$-divergence and is therefore jointly convex,~\eqref{eq:kl_cvxty_principle} directly follows from the convexity principle (Lemma~\ref{lem:convexity}).
    Next, for $q\in (0,\infty) \setminus \{1\}$, although $\Ren_q$ is not jointly convex~\cite[\S III-B]{van2014renyi} (for $q > 1$), it is an increasing transformation of a jointly convex $f$-divergence. Namely, define $g_q : \R_+\to\R_+$ via
    \begin{align*}
        g_q(s) \deq \frac{1}{q-1} \begin{cases}
            \log(1+s)\,, & q > 1\,, \\
            \log(1-s)\,, & q < 1\,,
        \end{cases}
    \end{align*}
    so that $\Ren_q = g_q(\Hell_q)$ (see~\eqref{eq:ren_and_hell}).
    Then, $g_q$ and $g_q^{-1}$ are increasing and $\Hell_q$ is jointly convex, it follows from Lemma~\ref{lem:convexity} that
    \begin{align*}
        \Ren_q(\mu P \mmid \nu P)
        = g_q(\Hell_q(\mu P \mmid \nu P))
        &\le g_q\Bigl( \inf_{\gamma \in \Coup(\mu,\nu)} \int \Hell_q(\delta_x P \mmid \delta_y P)\,\gamma(\D x, \D y)\Bigr) \\
        &\le \inf_{\gamma \in \Coup(\mu,\nu)} g_q\Bigl( \int g_q^{-1}(\rho(x,y)) \, \gamma(\D x, \D y)\Bigr)\,.
    \end{align*}
    This concludes the proof of~\eqref{eq:ren_cvxty_principle} by considering $q > 1$ and $q < 1$ separately.
\end{proof}

\begin{remark}[Optimal transport]\label{rem:convexity-ot}
    These regularity bounds can be viewed as reverse transport inequalities since the convexity principle naturally extracts an optimal transport cost in the regularity bound. The precise cost function is dictated by the regularity bound from Dirac initializations. For the Langevin SDE, the corresponding optimal transport costs are the $2$-Wasserstein distance for KL regularity, and a sub-Gaussian coupling cost for R\'enyi regularity (see Theorem~\ref{thm:langevin-disc} below). In the latter case, the coupling cost can also be related to the Orlicz--Wasserstein distance, c.f.~\cite{AltChe23warm}.
\end{remark}

\paragraph*{Refinements.}
In the application of the convexity principle to R\'enyi divergences above, we first applied an increasing transformation $g$ before invoking joint convexity.
The flexibility offered by such transformations sometimes leads to more refined bounds.
In general, if we can write a divergence $\msf D$ as a function $\msf D = g(\msf D')$ of some other jointly convex divergence $\msf D'$, where $g$ is \emph{strictly increasing} and \emph{convex}, then the bound obtained from applying the convexity principle to $\msf D'$ is stronger, as a consequence of Jensen's inequality $g(\int g^{-1}(\cdots)) \le \int (\cdots)$.

For example, in the case $q < 1$, it is known that the R\'enyi divergence $\Ren_q$ is jointly convex~\cite[\S III-B]{van2014renyi}, and hence we could have applied the convexity principle to $\Ren_q$ directly.
In the proof of Theorem~\ref{thm:renyi_cvxty_principle} above, we instead applied the convexity principle to $\Hell_q$, where $\Ren_q = g_q(\Hell_q)$ and $g$ is strictly increasing and convex, which therefore yields a sharper bound.

The gamut of potential transformations expands when we consider convexity in the first or second argument alone which, as we show below, can be combined with a joint convexity inequality to obtain new bounds.
In particular, we will apply this idea to $\Ren_q$, $q > 1$, based on the following two convexity statements:
\begin{enumerate}
    \item ${(\Hell_q + 1)}^{1/q}$ is convex in its first argument.
    Indeed, ${(\Hell_q+1)}^{1/q}(\mu \mmid \nu) = \norm{\frac{\D \mu}{\D\nu}}_{L^q(\nu)}$ is convex w.r.t.\ $\mu$ due to the convexity of the $L^q(\nu)$ norm.
    \item $\Ren_q$ is convex in its second argument~\cite[\S III-B]{van2014renyi}.
\end{enumerate}
This leads to two refined bounds which involve weak optimal transport costs~\cite{Gozetal17WeakOT}; c.f.~\cite{BacPam22WeakOT}.
The effect of these refinements will be explored in the next section. 

\begin{theorem}[Refined R\'enyi bounds]\label{thm:renyi_cvxty_principle_refined}
    Let $q > 1$ and let $P$ be a Markov kernel on a Polish space $\cX$.
    Let $\rho$ be a measurable function on $\cX\times \cX$ such that $\Ren_q(\delta_x P \mmid \delta_y P) \le \rho(x,y)$ for all $x,y\in\cX$.
    Then, the following two inequalities hold, where we write $\gamma_{1\mid 2}$ for the conditional distribution of the first coordinate given the second under $\gamma$ and similarly for $\gamma_{2\mid 1}$:
    \begin{align}\label{eq:refined_1}
        \Ren_q(\mu P \mmid \nu P)
        &\le \inf_{\gamma\in\Coup(\mu,\nu)} \frac{1}{q-1} \log \int\Bigl\{ \int \exp\bigl(\frac{q-1}{q}\,\rho(x,y)\bigr)\,\gamma_{1\mid 2}(\D x \mid y)\Bigr\}^q \, \nu(\D y)
    \end{align}
    and
    \begin{align}\label{eq:refined_2}
        \Ren_q(\mu P \mmid \nu P)
        &\le \inf_{\gamma\in\Coup(\mu,\nu)} \frac{1}{q-1} \log\int \exp\Bigl((q-1)\int \rho(x,y) \,\gamma_{2\mid 1}(\D y \mid x)\Bigr) \, \mu(\D x)\,.
    \end{align}
\end{theorem}
\begin{proof}
    Let $\gamma \in \Coup(\mu,\nu)$.
    For the first inequality, we write $\mu P = \int \delta_x P \,\gamma_{1\mid 2}(\D x \mid y)\,\nu(\D y)$, so that
    \begin{align*}
        (\Hell_q + 1)(\mu P \mmid \nu P)
        &\le \int (\Hell_q+1)\Bigl( \int \delta_x P \,\gamma_{1\mid 2}(\D x \mid y)\Bigm\Vert \delta_y P\Bigr) \, \nu(\D y) \\
        &\le \int \Bigl\{{(\Hell_q+1)}^{1/q}\Bigl( \int \delta_x P \,\gamma_{1\mid 2}(\D x \mid y)\Bigm\Vert \delta_y P\Bigr)\Bigr\}^q \, \nu(\D y) \\
        &\le \int \Bigl\{ \int {(\Hell_q+1)}^{1/q}(\delta_x P \mmid \delta_y P)\,\gamma_{1\mid 2}(\D x \mid y)\Bigr\}^q \, \nu(\D y)\,.
    \end{align*}
    For the second inequality, we write $\nu P = \int \delta_y P \,\gamma_{2\mid 1}(\D y \mid x)\,\mu(\D x)$, so that
    \begin{align*}
        (\Hell_q + 1)(\mu P \mmid \nu P)
        &\le \int (\Hell_q+1)\Bigl(\delta_x P \Bigm\Vert \int \delta_y P \,\gamma_{2\mid 1}(\D y \mid x)\Bigr) \, \mu(\D x) \\
        &= \int \exp\Bigl((q-1)\,\Ren_q\Bigl(\delta_x P \Bigm\Vert \int \delta_y P \,\gamma_{2\mid 1}(\D y \mid x)\Bigr)\Bigr) \, \mu(\D x) \\
        &\le \int \exp\Bigl((q-1)\int \Ren_q(\delta_x P \mmid \delta_y P) \,\gamma_{2\mid 1}(\D y \mid x)\Bigr) \, \mu(\D x)\,.
    \end{align*}
    The inequalities in the theorem statement follow.
\end{proof}

We provide dual versions of these arguments in \S\ref{app:dist_harnack}.

\subsection{Application to the Langevin diffusion}\label{ssec:disc:langevin}

Here we illustrate how the techniques developed in \S\ref{ssec:disc:scp}, \S\ref{ssec:disc:reduction}, \S\ref{ssec:disc:convexity} immediately yield tight regularity bounds for the Langevin SDE.
We discuss tightness of the bounds in \S\ref{app:tightness}.

\par In what follows, let ${(P_t)}_{t\ge 0}$ denote the semigroup corresponding to the Langevin SDE
\begin{align*}
    \D X_t = - \nabla V(X_t) \, \D t + \sqrt{2} \,\D B_t\,,
\end{align*}
where $B$ is a standard Brownian motion. It is a classical fact that under minimal assumptions, the law of $X_t$ converges to $\pi \propto \exp(-V)$, see e.g.,~\cite{bakry2014analysis} for background. Let $\hat P_h$ denote the Markov kernel corresponding to the discretized Langevin SDE with time step $h > 0$, i.e., $\hat P_h(x,\cdot) = Q_{2h}(x - h\,\nabla V(x), \cdot)$ where ${(Q_t)}_{t \geq 0}$ denotes the heat semigroup.

\begin{theorem}[Discrete-time regularity of Langevin]\label{thm:langevin-disc}
    Suppose that $\alpha I \preceq \nabla^2 V \preceq \beta I$ on $\R^d$. Define the shorthand $L  \deq  \max_{\lambda \in \{\alpha,\beta\}} |1 - h \lambda|$. Then
    \begin{align*}
        \KL(\mu \hat{P}_h^N \mmid \nu \hat{P}_h^N) \leq \frac{1-L^2}{4h\,(L^{-2N}-1)}\, W_2^2(\mu,\nu)
    \end{align*}
    and for any $q \in (0,1) \cup (1, \infty)$,
    \begin{align*}
        \Ren_q(\mu \hat{P}_h^N \mmid \nu \hat{P}_h^N) \leq \inf_{\gamma \in \Coup(\mu,\nu)} \frac{1}{q-1} \log \int \exp \Bigl( \frac{q\,(q-1)\,(1-L^2)}{4h\,(L^{-2N}-1)}\,\|x-y\|^2 \Bigr)\, \gamma(\D x, \D y)\,.
    \end{align*}
\end{theorem}
\begin{proof}
    It suffices to prove the discrete-time R\'enyi regularity bound between Dirac initializations:
    \begin{align}
        \Ren_q(\delta_x \hat{P}_h^N \mmid \delta_y \hat{P}_h^N) \leq \frac{q\, (1-L^2)}{4h\,(L^{-2N} -1)}\, \|x-y\|^2\,.\label{eq:reg-ren-disc-Dirac}
    \end{align}
    Indeed, the claim for arbitrary initializations then follows by the convexity principle in Theorem~\ref{thm:renyi_cvxty_principle}.

    \par To prove~\eqref{eq:reg-ren-disc-Dirac}, we use the one-to-multi-step reduction for regularity bounds (Theorem~\ref{thm:discrete-main}). To this end, we use the elementary fact that $\phi(x)  \deq  x - h\, \nabla V(x)$ is $L$-Lipschitz. Since $\hat{P}_h(x,\cdot) = Q_{2h}(\phi(x),\cdot)$, it follows that
    \begin{itemize}
        \item[(a)] $\hat{P}_h$ satisfies the $1$-step regularity bound~\eqref{eq:one_step_reg} with parameter $c \defeq \frac{qL^2}{4h}$. This follows by the identity for the R\'enyi divergence between Gaussians (Proposition~\ref{thm:renyi_prop}) and the Lipschitzness of the mapping $\phi$:
        \begin{align}
            \Ren_q (\delta_x \hat{P}_h \mmid \delta_y \hat{P}_h)
            = \frac{q\,\|\phi(x)-\phi(y)\|^2}{4h}
            \leq \frac{qL^2\,\|x-y\|^2}{4h}\,.\label{eq:langevin-1}
        \end{align}
        \item[(b)] $\hat{P}_h$ is $W_{\infty}$-Lipschitz with parameter $L$. This follows from a trivial coupling argument:
        \begin{align}
            W_{\infty}(\mu \hat{P}_h, \nu \hat{P}_h)
            \leq
            W_{\infty}(\phi_{\#} \mu, \phi_{\#} \nu)
            \leq
            L\, W_{\infty}(\mu,\nu)\,.
            \label{eq:langevin-2}
        \end{align}
    \end{itemize}
    Thus we may invoke Theorem~\ref{thm:discrete-main}. This proves the claim~\eqref{eq:reg-ren-disc-Dirac}. 
\end{proof}

This tight regularity result for the discretized Langevin semigroup immediately implies the following tight regularity result for the (standard, continuous-time) Langevin semigroup by taking the limit as the step size $h \searrow 0$ and the total elapsed continuous time is fixed to $T = Nh$. 

\begin{cor}[Continuous-time regularity of Langevin]\label{cor:langevin-cont}
    Suppose that $\alpha I \preceq \nabla^2 V$ on $\R^d$. Then
    \begin{align}\label{eq:langevin_kl}
        \KL(\mu P_T \mmid \nu P_T ) \leq \frac{\alpha}{2\,(\exp(2\alpha T) - 1)}\, W_2^2(\mu,\nu)
    \end{align}
    and for any $q \in (0,1) \cup (1, \infty)$,
    \begin{align}\label{eq:langevin_renyi}
        \Ren_q(\mu P_T  \mmid \nu P_T ) \leq \inf_{\gamma \in \Coup(\mu,\nu)} \frac{1}{q-1} \log \int \exp \Bigl( \frac{\alpha q\,(q-1)}{2\,(\exp(2\alpha T) - 1)}\,\|x-y\|^2 \Bigr)\, \gamma(\D x, \D y)\,.
    \end{align}
\end{cor}

Our bounds hold for any value of $\alpha \in \R$, provided that the expressions are interpreted accordingly. Namely, for $\alpha =0$, the occurrences of $\alpha/(\exp(2\alpha T) -1)$ in~\eqref{eq:langevin_kl} and~\eqref{eq:langevin_renyi} simplify to $1/(2T)$. Note that for our discrete-time results, we also need the upper bound $\nabla^2 V \preceq \beta I$ (which is standard for discretization analysis), but the dependence on $\beta$ vanishes as $h\searrow 0$.

We conclude this discussion with a few remarks.

\begin{remark}[Relationship with~\cite{AltTal23Langevin}]
    Theorem~\ref{thm:langevin-disc} (for $q \geq 1$) recovers the main result of~\cite{AltTal23Langevin} (see Remark A.4 therein) and strengthens it beyond $W_{\infty}$.
    We discuss the relationship of our work with the extant literature on differential privacy and sampling in \S\ref{ssec:discrete:dp}.
\end{remark}

\begin{remark}[Finiteness thresholds and refined R\'enyi regularity]\label{rem:finiteness}
    Unlike KL regularity, R\'enyi regularity can undergo a phase transition in which $\Ren_q(\mu P_T \mmid \nu P_T)$ becomes finite only after $T$ surpasses some threshold $T_0 > 0$. In contrast, the KL regularity is finite for arbitrarily small times as soon as the initial measures have finite second moment.
    \par The R\'enyi regularity bounds in Theorem~\ref{thm:langevin-disc} and Corollary~\ref{cor:langevin-cont}---while often exact for the OU process and its discretization, see \S\ref{app:tightness}---sometimes fail to tightly capture this finiteness threshold, in which case we turn toward the refined bounds of Theorem~\ref{thm:renyi_cvxty_principle_refined}.
    See \S\ref{app:tightness:finiteness}, where we explore the sharpness of the bounds on the finiteness threshold through various examples.
\end{remark}

\begin{theorem}[Refined R\'enyi regularity for Langevin]\label{thm:langevin-refined}
    Suppose $\alpha I \preceq \nabla^2 V$ on $\R^d$. For any $q > 1$,
     \begin{align}
        \Ren_q(\mu P_T \mmid \nu P_T)
        &\le \inf_{\gamma\in\Coup(\mu,\nu)} \frac{1}{q-1} \log \int\biggl( \int \exp\Bigl(\frac{\alpha\,(q-1)}{2\,(\exp(2 \alpha T) - 1)}\, \|x-y\|^2 \Bigr)\,\gamma_{1\mid 2}(\D x \mid y)\biggr)^q \, \nu(\D y)\,,\label{eq:langevin-regularity-refined-1}
    \end{align}
    and
    \begin{align}
        \Ren_q(\mu P_T \mmid \nu P_T)
        &\le \inf_{\gamma\in\Coup(\mu,\nu)} \frac{1}{q-1} \log\int \exp\biggl(\frac{\alpha q\,(q-1)}{2\,(\exp(2\alpha T) - 1)}\int \|x-y\|^2 \,\gamma_{2\mid 1}(\D y \mid x)\biggr) \, \mu(\D x)\,.\label{eq:langevin-regularity-refined-2}
    \end{align}
\end{theorem}
\begin{proof}
    Specialize the R\'enyi regularity bound~\eqref{eq:langevin_renyi} to Dirac initializations and apply the refined convexity principle for R\'enyi divergences (Theorem~\ref{thm:renyi_cvxty_principle_refined}).
\end{proof}

\subsection{Discussion: relationship with differential privacy and sampling}\label{ssec:discrete:dp}

Our development is motivated by the ``shifted divergence'' technique from differential privacy~\cite{pabi}, and in particular the modified version~\cite{AltTal22dp} which was recently used to established discrete mixing bounds~\cite{AltChe23warm, AltTal23Langevin}.
Briefly, that argument bounds the divergence $\cD$ (typically KL or R\'enyi) between the laws of two stochastic processes which evolve through the iterative, alternating application of additive noise (typically Gaussian) and a Lipschitz map (typically a step of gradient descent). That is, these arguments prove regularity results in the setting that $P = Q_1 Q_2$ where $Q_1$ is a convolution kernel and $Q_2$ is Wasserstein-Lipschitz~\cite{AltTal22dp, AltChe23warm, AltTal23Langevin}. The argument uses as a Lyapunov function the shifted divergence
\begin{align*}
    \cD^{(z)}(\mu \mmid \nu)  \deq  \inf_{\mu' \; : \; W_{\infty}(\mu,\mu') \leq z} \cD(\mu' \mmid \nu)\,,
\end{align*}
where $z \geq 0$ is a non-negative ``shift'' that allows changing the argument to the divergence in $W_{\infty}$ distance. The argument is
based on two key lemmas which track how this shifted divergence is affected by either additive noise $Q_1$ or a Wasserstein-Lipschitz kernel $Q_2$. Our framework generalizes, refines, and unifies this shifted divergence argument.

\paragraph*{Generality.} We identify the shifted composition rule (Theorem~\ref{thm:scr}) as the key information-theoretic principle that underlies the shifted divergence argument.
An important advantage of our level of generality is that our argument is no longer restricted to Markov kernels corresponding to additive noise. This generality is already manifest in Theorem~\ref{thm:discrete-main} which does not require decomposition of $P$ into the form $Q_1 Q_2$; and this freedom to choose more general ``shifts'' will be further exploited in future works in this series. 

\paragraph*{Refinement.} The convexity principle (\S\ref{ssec:disc:convexity}) yields improved regularity results (a.k.a., reverse transport inequalities) in which the Wasserstein distance is relaxed from $W_{\infty}$---e.g., KL bounds in terms of the initial $2$-Wasserstein distance $W_2$, and R\'enyi bounds in terms of the initial Orlicz--Wasserstein distance $W_{\psi_2}$. Although simple in hindsight, obtaining results beyond $W_{\infty}$ was a barrier in the literature for both differential privacy and sampling. For instance, this immediately improves and simplifies the discrete mixing results of discretized Langevin in both the overdamped~\cite{AltTal23Langevin} and underdamped settings~\cite{AltChe23warm}. For the previous, the improvement is for $W_{\infty}$ to $W_2$ or $W_{\psi_2}$ (corresponding to KL or R\'enyi, respectively); and for the latter this new argument further improves the constants in the regularity bound. 

\paragraph*{Unification.} Although the literature on differential privacy and sampling and the literature on diffusions have both sought to prove mixing/regularity bounds for stochastic processes, a high-level difference is that the former analyzes discrete-time processes while the latter analyzes continuous-time ones.
In this paper, we bridge the techniques from these communities by
constructing ``shifted processes'' that 1) explicitly realize the optimal shifts that are implicit in the shifted divergence argument, 2) extend to the continuous-time arguments of~\cite{ArnThaWan06HarnackCurvUnbdd} in the limit (details in \S\ref{sec:cont}).
	% !TEX root = ../scp1.tex

\section{Continuous-time arguments}\label{sec:cont}

In this section, we develop the continuous-time analogues of the proofs in \S\ref{sec:disc}; in particular, \S\ref{sec:cont:sync} develops the analogue of the synchronous coupling via Girsanov transformation, and \S\ref{sec:cont:opt} develops the analogue of the Wasserstein coupling via Otto calculus.
As we discuss below, the coupling in \S\ref{sec:cont:sync} was previously introduced in~\cite{ArnThaWan06HarnackCurvUnbdd} and subsequently used extensively in the literature, but the argument in \S\ref{sec:cont:opt} seems to be new.

For simplicity, we illustrate the techniques on the special case of the Langevin diffusion with semi-convex potential, i.e., $\nabla^2 V \succeq \alpha I$ for some $\alpha \in \R$.
We also assume that $\nabla V$ is Lipschitz continuous which, in light of the previous assumption, amounts to an upper bound on $\nabla^2 V$, although the upper bound does not enter into our quantitative results.
The Lipschitz continuity assumption is made for simplicity, to ensure that there is a unique strong solution to the Langevin SDE which is non-explosive.
Later, in \S\ref{ssec:ext:ito}, we consider the more general setting of uniformly elliptic It\^o diffusions.

\subsection{Synchronous coupling and Girsanov's theorem}\label{sec:cont:sync}

Fix $x,y\in\R^d$ and recall that our goal is to prove a bound on $\Ren_q(\delta_x P_T \mmid \delta_y P_T)$, where ${(P_t)}_{t\ge 0}$ is the Langevin semigroup (and $q=1$ corresponds to $\Ren_q = \KL$).
The natural continuous-time analogue of the synchronous coupling in \S\ref{sec:disc} is to first define the processes
\begin{align*}
    \D X_t
    &= -\nabla V(X_t) \,\D t + \sqrt 2 \, \D B_t\,, &X_0 &= x\,, \\
    \D Y_t
    &= -\nabla V(Y_t) \, \D t + \sqrt 2 \, \D B_t\,, &Y_0 &= y\,,
\end{align*}
so that $\law(X_T) = \delta_x P_T$ and $\law(Y_T) = \delta_y P_T$.
However, instead of bounding the divergence between the laws of $\{X_t\}_{t\in [0,T]}$ and $\{Y_t\}_{t\in [0,T]}$, we instead introduce an auxiliary process $\{X_t'\}_{t\in [0,T]}$ such that $X_T' = Y_T$ almost surely (hence $\law(X_T') = \law(Y_T)$) of the form
\begin{align}\label{eq:sync_aux_cont}
    \D X_t'
    &= \{-\nabla V(X_t') +\eta_t\,(Y_t - X_t')\} \, \D t + \sqrt 2\,\D B_t\,, \qquad X_0' = x\,,
\end{align}
where $\{\eta_t\}_{t\in [0,T]}$ is a deterministic and non-negative process.
The divergence between the laws of $\{X_t\}_{t\in [0,T]}$ and $\{X_t'\}_{t\in [0,T]}$ can then be bounded by Girsanov's theorem.
Such a coupling was first introduced in~\cite{ArnThaWan06HarnackCurvUnbdd} and subsequently used to establish Harnack and reverse transport inequalities for a bevy of settings, including for diffusions on Riemannian manifolds~\cite{ArnThaWan09HarnackNoncompact, Wang14Diffusion}, for diffusions with multiplicative noise~\cite{Wang11HarnackMultNoise, WanYua11HarnackFuncMult}, under low regularity~\cite{Shao13HarnackSingular, HuaZha19HarnackDini, ZhaYua21Zvonkin}, for SPDEs~\cite{Wang07HarnackPorous, LiuWan08HarnackFastDiffusion, DaPRocWan09HilbertHarnack, EsSvonSch09HarnackMemory, Liu09HarnackMonotone, Zhang10SPDEsReflectionHarnack, Ouy11HarnackMultivalued, WanYua11HarnackFuncMult, Wang13HarnackSPDE, WanZha13DerivDegen}, for distribution-dependent processes~\cite{Wang18Landau, HuaWan19DistDepSingular, HuaWan22SingularMKV}, for jump processes~\cite{Wang11CouplingJumps, OuyRocWan12HarnackOUJump, WanZha15HarnackStable}, and for SDEs driven by fractional Brownian motion~\cite{Fan15FBM}; see~\cite{Wang12Coupling} for a survey.
For completeness, we sketch the argument below, focusing on the KL divergence bound for simplicity. The extension to R\'enyi divergences is given in \S\ref{app:renyi}.

There are few other ways to construct this auxiliary processes. Related synchronous constructions are discussed briefly at the end of this section, and in the next section, we show that the Wasserstein coupling approach of \S\ref{sec:disc} leads to a distinct continuous-time interpretation.

\paragraph*{Reverse transport inequality for the KL divergence.}
Since our final bound will depend only on $\norm{x-y}$, it does not matter whether we bound $\KL(\delta_x P_T \mmid \delta_y P_T)$ or $\KL(\delta_y P_T \mmid \delta_x P_T)$; for convenience we bound the latter.
The key idea is to realize the auxiliary process~\eqref{eq:sync_aux_cont} via a Girsanov transformation of the Wiener measure.
To do so, let $\{B_t'\}_{t\ge 0}$ be a standard Brownian motion under the path measure $\PathAux_T$ on $\mc C([0,T];\R^d)$ and consider the solution to the coupled system of SDEs
\begin{align}\label{eq:sync_cont_coupled}
\begin{aligned}
    \D X_t
    &= \{-\nabla V(X_t) + \eta_t\,(Y_t - X_t)\}\, \D t + \sqrt 2 \, \D B_t'\,, \quad &X_0 &= x\,, \\
    \D Y_t
    &= -\nabla V(Y_t) \, \D t + \sqrt 2 \, \D B_t'\,, \quad &Y_0 &= y\,.
\end{aligned}
\end{align}
Then, let $\{B_t\}_{t\in [0,T]}$ be such that
\begin{align*}
    \D X_t
    &= -\nabla V(X_t) \, \D t + \sqrt 2 \,\D B_t\,,
\end{align*}
i.e., $\D B_t = \D B_t' + \frac{\eta_t}{\sqrt 2}\,(X_t - Y_t) \, \D t$.
If we define the $\PathAux_T$-martingale $t\mapsto M_t \deq -\int_0^t \frac{\eta_s}{\sqrt 2}\,\langle X_s - Y_s, \D B_s'\rangle$, and if $\exp(M - \frac{1}{2}\,[M,M])$ is a martingale (rather than merely a local martingale), Girsanov's theorem (see~\cite[Theorem 5.22]{legall2016stochasticcalc}) ensures that under the path measure $\PathA_T$ defined via
\begin{align}\label{eq:sync_cont_girsanov}
    \frac{\D \PathA_T}{\D\PathAux_T}
    &= \exp\bigl(M_T - \frac{1}{2}\,{[M,M]}_T\bigr)\,,
\end{align}
the process $\{B_t\}_{t\in [0,T]}$ is a standard Brownian motion.
These path measures are defined so that under $\PathA_T$, $\law(X_T) = \delta_x P_T$, whereas under $\PathAux_T$, if $X_T = Y_T$ almost surely, then $\law(X_T) = \delta_y P_T$.
It follows that
\begin{align}\label{eq:sync_cont_kl}
    \KL(\delta_y P_T \mmid \delta_x P_T)
    &\le \KL(\PathAux_T \mmid \PathA_T)
    = -\E_{\PathAux_T} \log \frac{\D\PathA_T}{\D\PathAux_T}
    = \frac{1}{4}\, \E_{\PathAux_T}\int_0^T \eta_t^2 \,\norm{X_t-Y_t}^2 \, \D t\,.
\end{align}

Next, since
\begin{align*}
    \D (X_t - Y_t)
    &= \{-\nabla V(X_t) + \nabla V(Y_t) + \eta_t\,(Y_t - X_t)\}\, \D t\,,
\end{align*}
hence by It\^o's formula and semi-convexity,
\begin{align*}
    \D \norm{X_t - Y_t}^2
    &= -2\,\langle X_t - Y_t, \nabla V(X_t) - \nabla V(Y_t) + \eta_t\,(X_t - Y_t)\rangle \, \D t
    \le -2\,(\alpha + \eta_t)\,\norm{X_t - Y_t}^2 \, \D t\,,
\end{align*}
and therefore by Gr\"onwall's lemma,
\begin{align}\label{eq:sync_cont_dist}
    \norm{X_t - Y_t}^2
    &\le \exp\Bigl(-2\alpha t - 2\int_0^t \eta_s \, \D s\Bigr)\,\norm{x-y}^2\,.
\end{align}
Substituting this into~\eqref{eq:sync_cont_kl}, we find that
\begin{align}\label{eq:sync_cont_kl_2}
    \KL(\delta_y P_T \mmid \delta_x P_T)
    &\le \frac{\norm{x-y}^2}{4} \int_0^T \eta_t^2 \exp\Bigl(-2\alpha t - 2\int_0^t \eta_s \, \D s\Bigr) \, \D t\,.
\end{align}
We now make the optimal choice $\eta_t = 2\alpha/\{\exp(2\alpha\,(T-t))-1\}$ (this should be interpreted as $\eta_t = 1/(T-t)$ when $\alpha = 0$); in \S\ref{app:calc}, we show how this expression can be derived using the calculus of variations.
With this choice, $\int_0^t \eta_s \, \D s = \log \frac{1-\exp(-2\alpha T)}{1-\exp(-2\alpha\,(T-t))}$ (or $\log\frac{T}{T-t}$ for $\alpha = 0$).
In particular, from~\eqref{eq:sync_cont_dist}, we see that $X_t - Y_t \to 0$ almost surely as $t\nearrow T$, as required.
Finally, substitution into~\eqref{eq:sync_cont_kl_2} establishes the optimal reverse transport inequality
\begin{align*}
    \KL(\delta_y P_T \mmid \delta_x P_T)
    &\le \frac{\alpha\,\norm{x-y}^2}{2\,(\exp(2\alpha T) - 1)}\,,
\end{align*}
up to a few technical details which we address in the subsequent remark.

\begin{remark}[Technical]\label{rmk:sync_cont_technical}
    The above proof sketch is  rigorous aside from a few issues which we discuss here.
    First, since our eventual choice of $\eta_t$ blows up as $t\nearrow T$, the existence and uniqueness of the system~\eqref{eq:sync_cont_coupled} on $[0,T]$ does not follow from the basic theory of SDEs.
    However, the system is well-posed on $[0,T-\varepsilon]$ for every $\varepsilon > 0$.
    Therefore, this issue is easily remedied by noting that $\law_{\PathA_T}(X_t) \to \delta_x P_T$ and $\law_{\PathAux_T}(X_t) \to \delta_y P_T$ weakly as $t\nearrow T$ and appealing to the joint lower semicontinuity of the KL divergence.

    Similarly, in order for~\eqref{eq:sync_cont_girsanov} to define a valid probability measure $\PathA_T$, the Girsanov factor $t\mapsto \mc E_t \deq \exp(M_t - \frac{1}{2}\,{[M,M]}_t)$ must be a valid martingale, which amounts to exponential integrability of the quantity $\int_0^T \norm{X_t-Y_t}^2 \, \D t$.
    However, this can also be avoided by considering a localizing sequence of stopping times ${(\tau_k)}_{k\in\N}$ such that $\tau_k\nearrow \infty$ almost surely, and $X_t$ and $Y_t$ are bounded for $t\le \tau_k$.
    Then, the stopped process $\mc E_{\cdot \wedge \tau_k}$ is a valid martingale for each $k$, and we can again appeal to the joint lower semicontinuity of the KL divergence.
    Since this type of argument is standard in the literature, the details are omitted for brevity.
\end{remark}

We conclude this section with a discussion of related synchronous constructions.

\begin{remark}[Alternative construction]
    In some works (see, e.g.,~\cite{Wang12Coupling}), a slightly different form is considered for the added drift, namely one adds $\eta_t \, \frac{Y_t - X_t'}{\norm{Y_t - X_t'}}$ instead of $\eta_t \, (Y_t - X_t')$.
    Both approaches can be used to derive sharp R\'enyi regularity bounds for the Langevin SDE\@, but as noted in~\cite{Wang11HarnackMultNoise}, it is necessary to consider an unbounded drift as $t\nearrow T$ in order to handle the multiplicative noise case in \S\ref{ssec:ext:ito}.
    For concreteness, we stick with the latter form of the drift.
\end{remark}

\begin{remark}[Coupling with a deterministic shift]
    The coupling in~\eqref{eq:sync_aux_cont} adds a \emph{random} drift to the auxiliary process in order to force it to hit another process by time $T$.
    However, there is another method in which we simply define the auxiliary process to satisfy $X' = X + v$, where $v : [0,T]\to\R^d$ is a deterministic curve.
    To distinguish it from the synchronous coupling, we refer to the latter method as \emph{coupling with a deterministic shift}.
    Coupling with a deterministic shift has been used to derive Bismut-type derivative formulas (c.f.~\cite[\S 1.1.1]{Wang13HarnackSPDE}), which in turn can be used to establish power Harnack inequalities; however, to the best of our knowledge, the resulting Harnack inequalities are typically not sharp, unlike the ones derived via synchronous coupling.
    Interestingly, as noted in~\cite{Wang14ShiftHarnack}, coupling with a deterministic shift yields regularity for Kolmogorov's \emph{forward} equation, in contrast to the Harnack inequalities considered here which encode regularity for Kolmogorov's \emph{backward} equation (see \S\ref{ssec:apps:background} for further discussion).
    The forward regularity problem and its information-theoretic reformulation will be explored in a forthcoming work.
\end{remark}

\begin{remark}[Relationship with the F\"ollmer drift]
    The coupling~\eqref{eq:sync_aux_cont} can be interpreted as follows: we add a drift $t\mapsto b_t \deq \eta_t\,(Y_t - X_t')$ to the Langevin diffusion to ensure that the process has law $\delta_y P_T$ at time $T$.
    By a similar argument based on Girsanov's theorem, any adapted and well-behaved drift ${(b_t)}_{t\in [0,T]}$ with this property leads to the bound 
    \[
    \KL(\delta_y P_T \mmid \delta_x P_T) \le \frac{1}{4} \int_0^T \E[\norm{b_t}^2] \, \D t\,.
    \]
    It is then natural to ask what the \emph{optimal} drift is.
    The answer is the \emph{F\"ollmer drift}, given by $b_t^\star = 2\,\nabla \log P_{T-t}\frac{\D \delta_y P_T}{\D \delta_x P_T}(X_t')$, which makes the above inequality hold with equality~\cite{Fol1985Reversal}.
    Despite the recent success of the F\"ollmer process for establishing functional inequalities (see, e.g.,~\cite{Bor00Diffusion, Lehec13Follmer, CatGui14SemiLC, EldLee18Reg, MikShe21BrownianTransport} and the connection with stochastic localization~\cite{klaput21spectral}) and its appealing optimality property, it seems less tractable for the purpose of establishing reverse transport inequalities, as we are interested in doing here.
\end{remark}

\subsection{Wasserstein coupling and Otto calculus}\label{sec:cont:opt}

We now introduce the continuous-time analogue of the Wasserstein coupling argument in \S\ref{sec:disc}.
Unlike the synchronous coupling discussed in the previous section, which was based on path space arguments (notably, through the use of Girsanov's theorem), the present approach is more closely tied with the theory of optimal transport.

Again, fix $x,y\in\R^d$, and for ease of notation write $\mu_t \deq \delta_x P_t$ and $\nu_t \deq \delta_y P_t$.
We define a surrogate process $\{\mu_t'\}_{t\in [0,T]}$ such that $\mu_0' = \nu_0$ and $\mu_T' = \mu_T$ as follows: let $X_0' \sim \nu_0$ be a random variable that evolves according to the ODE
\begin{align}\label{eq:cont_ot_aux}
    \dot X_t'
    &= -\nabla \log \frac{\mu_t'}{\pi}(X_t') + \eta_t\,(T_{\mu_t'\to\mu_t}-{\id})(X_t')\,,
\end{align}
where $\mu_t' \deq \law(X_t')$ and $T_{\mu_t'\to\mu_t}$ denotes the optimal transport map from $\mu_t'$ to $\mu_t$.
To interpret this equation, recall that $\nabla \log \frac{\mu_t'}{\pi}$ is the Wasserstein gradient of the KL divergence $\KL(\cdot \mmid \pi)$ at $\mu_t'$ (see~\cite[\S 10.4]{AGS}).
Thus, the dynamics $\dot X_t' = -\nabla \log \frac{\mu_t'}{\pi}(X_t')$ yields the Wasserstein gradient flow of the relative entropy, which was shown in the work of R.\ Jordan, D.\ Kinderlehrer, and F.\ Otto~\cite{jordan1998variational} to describe the evolution of the marginal law of the Langevin diffusion.
The dynamics~\eqref{eq:cont_ot_aux} adds onto the Wasserstein gradient flow an additional term which drives the auxiliary process towards the original process $\{\mu_t\}_{t\in [0,T]}$.

More precisely, through the description of solutions to the continuity equation (see~\cite[\S 8]{AGS}), the process~\eqref{eq:cont_ot_aux} leads to the following PDE in the space of measures, which holds in the weak sense (in duality with space-time test functions):
\begin{align}\label{eq:cont_ot_aux_measures}
    \partial_t \mu_t'
    &= \divergence\bigl(\mu_t'\,\bigl(\nabla \log \frac{\mu_t'}{\pi} - \eta_t \,(T_{\mu_t'\to\mu_t}-{\id})\bigr)\bigr)\,.
\end{align}

We next show how the auxiliary dynamics~\eqref{eq:cont_ot_aux_measures} can also be used to reach optimal reverse transport inequalities; to the best of our knowledge, this argument is new.
To avoid obfuscating the flow of ideas with technical details, we keep our discussion at a formal level, i.e., we do not elaborate on the approximation arguments needed to make the following proof fully rigorous. For brevity, we also just focus on the KL bound here and defer the extenstion to R\'enyi divergence to \S\ref{app:renyi}.
\par Analogously to Theorem~\ref{thm:discrete-main}, the proof is based on two key bounds: on the KL divergence and the Wasserstein distance.

\paragraph*{Divergence bound for the auxiliary process.}
We first differentiate $t\mapsto \KL(\mu_t' \mmid \nu_t)$, where both arguments evolve simultaneously in time. This calculation is based on the simultaneous differentiation of KL divergence when both processes are evolving, a trick that has been used in other contexts in~\cite{VempalaW19, chenetal2022proximalsampler}.
We recall the calculation here for convenience.
\begin{align}
	\partial_t \KL(\mu_t'\mmid \nu_t)
	&= \int \partial_t \bigl( \mu_t'\log \frac{\mu_t'}{\nu_t}\bigr)
	= \int (\partial_t \mu_t') \log \frac{\mu_t'}{\nu_t} + \int \mu_t'\, \bigl(\frac{\partial_t \mu_t'}{\mu_t'} - \frac{\partial_t\nu_t}{\nu_t}\bigr) \nonumber\\
	&= -\int \mu_t' \,\bigl\langle \nabla \log \frac{\mu_t'}{\nu_t}, \nabla \log\frac{\mu_t'}{\pi} - \eta_t \,(T_{\mu_t'\to\mu_t}-{\id})\bigr\rangle 
	+ \int \nu_t\,\bigl\langle \nabla \frac{\mu_t'}{\nu_t}, \nabla \log \frac{\nu_t}{\pi}\bigr\rangle \nonumber\\
	&= -\int \mu_t' \,\bigl\langle \nabla \log \frac{\mu_t'}{\nu_t}, \nabla \log\frac{\mu_t'}{\pi} - \eta_t \,(T_{\mu_t'\to\mu_t}-{\id})\bigr\rangle 
	+ \int \mu_t'\,\bigl\langle \nabla\log \frac{\mu_t'}{\nu_t}, \nabla \log \frac{\nu_t}{\pi}\bigr\rangle \nonumber\\
	&= -\int \mu_t' \,\bigl\lVert \nabla \log \frac{\mu_t'}{\nu_t}\bigr\rVert^2 + \eta_t \int \mu_t'\,\bigl\langle \nabla \log \frac{\mu_t'}{\nu_t}, T_{\mu_t'\to\mu_t}-{\id}\bigr\rangle \nonumber\\
	&\le \frac{\eta_t^2}{4} \,W_2^2(\mu_t,\mu_t')\,, \label{eq:cont_ot_kl}
\end{align}
where the last inequality follows from the Cauchy{--}Schwarz inequality, Young's inequality, and $\norm{T_{\mu_t'\to\mu_t}-{\id}}_{L^2(\mu_t')} = W_2(\mu_t,\mu_t')$.

\paragraph*{Distance bound for the auxiliary process.}
We next differentiate $t\mapsto W_2^2(\mu_t,\mu_t')$.
Invoking~\cite[Theorem 23.9]{villani2009optimal},
\begin{align}
    \frac{1}{2}\,\partial_t W_2^2(\mu_t,\mu_t')
    &= \int \bigl\langle T_{\mu_t\to\mu_t'} - {\id}, \nabla \log \frac{\mu_t}{\pi} \bigr\rangle \, \D \mu_t + \int \bigl\langle T_{\mu_t'\to\mu_t}-{\id}, \nabla \log \frac{\mu_t'}{\pi} - \eta_t\,(T_{\mu_t'\to\mu_t} - {\id})\bigr\rangle\,\D \mu_t' \nonumber
    \\
    &= 
    \int \bigl\langle T_{\mu_t\to\mu_t'} - {\id}, \nabla \log \frac{\mu_t}{\pi} \bigr\rangle \, \D \mu_t + \int \bigl\langle T_{\mu_t'\to\mu_t}-{\id}, \nabla \log \frac{\mu_t'}{\pi} \bigr\rangle\,\D \mu_t' - \eta_t W_2^2(\mu_t, \mu_t')\,.
    \label{eq:deriv_w2}
\end{align}
From the $\alpha$-geodesic convexity of the KL divergence (see~\cite[Particular Case 23.15]{villani2009optimal}),
\begin{align*}
    \KL(\mu_t \mmid \pi)
    &\ge \KL(\mu_t' \mmid \pi) + \int \bigl\langle \nabla \log \frac{\mu_t'}{\pi}, T_{\mu_t'\to\mu_t}-{\id}\bigr\rangle \, \D \mu_t' + \frac{\alpha}{2}\,W_2^2(\mu_t,\mu_t')\,, \\
    \KL(\mu_t' \mmid \pi)
    &\ge \KL(\mu_t \mmid \pi) + \int \bigl\langle \nabla \log \frac{\mu_t}{\pi}, T_{\mu_t\to\mu_t'}-{\id}\bigr\rangle \, \D \mu_t + \frac{\alpha}{2}\,W_2^2(\mu_t,\mu_t')\,.
\end{align*}
Summing these two inequalities and combining with~\eqref{eq:deriv_w2}, we obtain
\begin{align*}
    \partial_t W_2^2(\mu_t, \mu_t')
    &\le -2\,(\alpha+\eta_t)\,W_2^2(\mu_t,\mu_t')\,.
\end{align*}
By Gr\"onwall's inequality,
\begin{align}\label{eq:cont_ot_w2}
    W_2^2(\mu_t,\mu_t')
    &\le \exp\Bigl(-2\alpha t - 2\int_0^t \eta_s \, \D s\Bigr)\,\norm{x-y}^2\,.
\end{align}

\paragraph*{Concluding the argument.}
Substituting~\eqref{eq:cont_ot_w2} into~\eqref{eq:cont_ot_kl} yields
\begin{align*}
    \KL(\delta_x P_T \mmid \delta_y P_T)
    &= \KL(\mu_T' \mmid \nu_T)
    \le \frac{\norm{x-y}^2}{4} \int_0^T \eta_t^2 \exp\Bigl(-2\alpha t - 2\int_0^t \eta_s \, \D s\Bigr)\,\D t\,.
\end{align*}
Note that this leads to the same bound as the one obtained in \S\ref{sec:cont:sync}. In particular, with the same choice of $\eta_t = 2\alpha/\{\exp(2\alpha\,(T-t))-1\}$, we once again arrive at the reverse transport inequality
\begin{align*}
    \KL(\delta_x P_T \mmid \delta_y P_T)
    &\le \frac{\alpha\,\norm{x-y}^2}{2\,(\exp(2\alpha T) - 1)}\,.
\end{align*}
\par We conclude this section with a remark on a different method of organizing the calculations that leads to a link with the JKO scheme of~\cite{jordan1998variational}.

\begin{remark}[Connection with the JKO scheme]
Through~\eqref{eq:cont_ot_kl} and~\eqref{eq:cont_ot_w2}, we have bounded the derivative in time for the KL divergence and Wasserstein distance separately; however, since these derivatives are expressed in terms of the same quantity $W_2^2(\mu_t,\mu_t')$, it also leads to the decay of a joint Lyapunov functional which incorporates both terms simultaneously. 
Namely, define
\begin{align*}
    \mc L_t
    &\deq \KL(\mu_t' \mmid \pi) + \frac{1}{2\lambda_t}\,W_2^2(\mu_t,\mu_t')\,.
\end{align*}
Differentiating this quantity and using~\eqref{eq:cont_ot_kl},~\eqref{eq:cont_ot_w2} yields
\begin{align*}
	\dot{\mc L}_t
	&\le \bigl( \frac{\eta_t^2}{4} - \frac{\dot\lambda_t}{2\lambda_t^2} - \frac{\alpha + \eta_t}{\lambda_t}\bigr)\,W_2^2(\mu_t,\mu_t')\,.
\end{align*}
Now we can optimize over the choice of $\eta_t$, leading to $\eta_t = 2/\lambda_t$, and then $\dot{\mc L}_t \le 0$ provided that $\dot\lambda_t +2\,(\alpha\lambda_t + 1) \ge 0$.
If we solve this differential inequality, enforcing that $\lambda_t \to 0$ as $t\nearrow T$, it then leads to the choice $\lambda_t = \{\exp(2\alpha \,(T-t))-1\}/\alpha$.

This result can be reinterpreted as follows.
For $\lambda > 0$, define
\begin{align}\label{eq:moreau}
    \MEKL(\mu; \lambda) \deq \inf_{\mu'\in \mc P_2(\R^d)}\Bigl\{\KL(\mu'\mmid \pi) + \frac{1}{2\lambda}\,W_2^2(\mu,\mu')\Bigr\}\,.
\end{align}
This is the generalization to the Wasserstein space of the Moreau--Yosida envelope, which is classically studied in conjunction with proximal methods in optimization~\cite{rockafellar1997convex, Beck17FirstOrder}.
The Moreau--Yosida envelope also appears as the Hopf--Lax solution to the Hamilton--Jacobi equation in classical mechanics.
In the Wasserstein space, the proximal point algorithm for minimizing the KL divergence (the iterates of which are generated by successively minimizing~\eqref{eq:moreau}) is commonly referred to as the ``minimizing movements scheme'' or the ``JKO scheme''. With our choice of $\{\lambda_t\}_{t\in [0, T]}$, for any $t < T$,
\begin{align}\label{eq:ME_as_lyapunov}
    \MEKL(\mu_t;\lambda_t)
    \le \mc L_t
    \le \mc L_0
    = \MEKL(\mu_0; \lambda_0)
\end{align}
where the last equality holds if we choose\footnote{In the calculations above, we worked with $\mu_0 = \delta_x$ and $\nu_0 = \delta_y$ for simplicity and consistency with previous sections. However, it is clear that the same calculations go through for general initial conditions $\mu_0$, $\nu_0$, with the only modification being that we replace $\norm{x-y}^2$ with $W_2^2(\mu_0,\nu_0)$.} $\nu_0$ to be the minimizer in $\MEKL(\mu_0;\lambda_0)$.
In particular,~\eqref{eq:ME_as_lyapunov} readily implies
\begin{align*}
    \KL(\mu_T \mmid \pi)
    &\le \liminf_{t\nearrow T} \MEKL(\mu_t;\lambda_t)
    \le \MEKL(\mu_0;\lambda_0)
    \le \frac{\alpha\,W_2^2(\mu_0,\pi)}{2\,(\exp(2\alpha T)-1)}\,.
\end{align*}
The moral of the story, then, is that \emph{the Moreau{--}Yosida envelope with time-varying parameter $\lambda$ can be used as a Lyapunov functional for the gradient flow}.
This seems to be a new observation, and the use of this principle as a unifying analysis framework for optimization will be explored in a separate work.
\end{remark}
    % !TEX root = ../scp1.tex

\section{Extensions to other settings}\label{sec:ext}

In this section, we consider extensions of our results to a few different settings, focusing on KL regularity throughout for simplicity.

\subsection{Diffusions on manifolds}\label{ssec:ext:riem}

We briefly note that the proofs in \S\ref{sec:cont} readily extend to the setting of a diffusion on a complete Riemannian manifold $\cM$.
Suppose that the generator of the diffusion is $\Delta - \langle \nabla V, \nabla \cdot\rangle$ where $\Delta$ is the Laplace{--}Beltrami operator, and that the condition $\text{Ric}_{\cM} + \nabla^2 V \succeq \alpha$ holds for some $\alpha\in\R$.
The coupling introduced in \S\ref{sec:cont:sync} was in fact originally developed to prove Harnack inequalities in the Riemannian setting in~\cite{ArnThaWan06HarnackCurvUnbdd}; see~\cite[Theorem 2.3.2]{Wang14Diffusion} for general definitions. As for the optimal transport approach of \S\ref{sec:cont:opt}, the proofs go through as before using the calculus of optimal transport over Riemannian manifolds, see~\cite[\S 23]{villani2009optimal}.

\subsection{Multiplicative noise}\label{ssec:ext:ito}

In this section, we consider the more general It\^o SDE
\begin{align*}
    \D X_t = b_t(X_t) \, \D t + \sigma_t(X_t) \, \D B_t\,,
\end{align*}
and we assume that the SDE is well-posed.
In the paper~\cite{Wang11HarnackMultNoise}, F.-Y.\ Wang obtained a log-Harnack inequality under the following assumptions.

\begin{ass}\label{ass:mult_noise}
    The following hold.
    \begin{itemize}
        \item ($W_2$-Lipschitz) There exists $\alpha \in\R$ such that for all $x,y\in\R^d$ and $t\in [0,T]$,
        \begin{align*}
            \langle b_t(x) - b_t(y), x-y\rangle + \frac{1}{2}\,\norm{\sigma_t(x) - \sigma_t(y)}_{\rm HS}^2
            &\le -\alpha\,\norm{x-y}^2\,.
        \end{align*}
        \item (Uniformly elliptic) There exists $\lambda > 0$ such that for all $x\in\R^d$ and $t \in [0,T]$,
        \begin{align*}
            \sigma_t(x)\,{\sigma_t(x)}^\T \succeq \lambda I\,.
        \end{align*}
    \end{itemize}
\end{ass}

Here, we illustrate that his result can be obtained in a simple manner via our techniques. In fact, we will do so by obtaining a reverse transport inequality for the discretization
\begin{align}\label{eq:discretized_ito}
    \hat X_{(n+1)h} = \hat X_{nh} + h\,b_{nh}(\hat X_{nh}) + \sqrt h\,\sigma_{nh}(\hat X_{nh})\,\xi_{nh}\,,
\end{align}
where $\{\xi_{nh}\}_{n\in\N}$ is an i.i.d.\ sequence of standard Gaussian vectors, and then passing to the limit $h\searrow 0$.
To the best of our knowledge, our result for the discretization~\eqref{eq:discretized_ito} is new.

For the discretization, we must also impose an additional assumption.
In our final bound, however, the dependence on the parameters $\beta$, $\Lambda$ below will vanish when we take $h\searrow 0$.

\begin{ass}\label{ass:mult_two}
    There exist $\beta, \Lambda > 0$ such that for all $x,y\in\R^d$ and $t\in [0,T]$,
    \begin{align*}
        \norm{b_t(x) - b_t(y)} \le \beta\,\norm{x-y} \qquad\text{and}\qquad\sigma_t(x)\,{\sigma_t(x)}^\T \preceq \Lambda I\,.
    \end{align*}
\end{ass}

We can now state our main result for this section.

\begin{theorem}\label{thm:mult_noise}
    Let $\{\hat\mu_{nh}\}_{n=0}^N$ and $\{\hat\nu_{nh}\}_{n=0}^N$ denote the marginal laws of the process~\eqref{eq:discretized_ito} started from $x$ and from $y$ respectively.
    Suppose that Assumptions~\ref{ass:mult_noise} and~\ref{ass:mult_two} hold with $T = Nh$.
    Then,
    \begin{align*}
        \KL(\hat\mu_{Nh} \mmid \hat\nu_{Nh})
        &\le \Bigl(1 + \frac{4\,(\beta-\alpha)\,(\Lambda/\lambda)^3\, h}{L^2}\Bigr)\,\frac{\alpha\,(1-\beta h/2)}{\lambda\, (L^{-2N} - 1)}\,\norm{x-y}^2\,, \qquad\text{where}~L^2 \deq 1-2\alpha h + \beta^2 h^2\,.
    \end{align*}
\end{theorem}

In particular, if we let $h\searrow 0$ with $Nh\to T$, then with the obvious notation,
\begin{align*}
    \KL(\mu_T \mmid \nu_T)
    &\le \frac{\alpha \,\norm{x-y}^2}{\lambda \, (\exp(2\alpha T) - 1)}\,.
\end{align*}
This recovers the result of~\cite{Wang11HarnackMultNoise}, at least for the log-Harnack inequality.
It also includes the results for the Langevin diffusion as a special case.
We do not treat the other Harnack inequalities corresponding to R\'enyi regularity here, as the use of R\'enyi divergences introduces substantial new complications in this setting.

We now give the proof, which simply amounts to checking the two conditions of Theorem~\ref{thm:discrete-main}.

\begin{proof}[Proof of Theorem~\ref{thm:mult_noise}]
    Let $t\in [0,T]$ and consider the kernel $P$ such that for any $x\in\R^d$, $\delta_x P = \cN(x+h\,b_t(x), \; h\,\sigma_t(x)\,{\sigma_t(x)}^\T)$.
    Write $\Sigma_t(x) \deq \sigma_t(x) \,\sigma_t(x)^\T$ and $\Sigma_t(y) \deq \sigma_t(y) \,\sigma_t(y)^\T$ for simplicity.

    \textbf{One-step regularity.}
    Using the closed-form expression for the KL divergence between Gaussians, we can compute
    \begin{align*}
        \KL(\delta_x P \mmid \delta_y P)
        &= \frac{1}{2}\,\Bigl[\frac{1}{h}\,\bigl\langle \Sigma_t(y)^{-1}, (x+h\,b_t(x) - y - h\,b_t(y))^{\otimes 2}\bigr\rangle + \tr f\bigl(\Sigma_t(y)^{-1/2}\,\Sigma_t(x)\, \Sigma_t(y)^{-1/2}\bigr)\Bigr]\,,
    \end{align*}
    where $f$ is the mapping $\zeta \mapsto \zeta - 1 - \log \zeta$.

    For the first term, we use the uniform ellipticity $\Sigma_t(y) \succeq \lambda I$, so it suffices to bound the quantity $\norm{x+h\,b_t(x) - y-h\,b_t(y)}^2$.
    Expanding the square and applying Assumptions~\ref{ass:mult_noise} and~\ref{ass:mult_two},
    \begin{align*}
        \norm{x+h\,b_t(x)-y-h\,b_t(y)}^2
        &= \norm{x-y}^2 +2h\,\langle b_t(x) - b_t(y), x-y\rangle + h^2\,\norm{b_t(x) - b_t(y)}^2 \\
        &\le (1-2\alpha h + \beta^2 h^2)\,\norm{x-y}^2\,.
    \end{align*}

    For the second term, let $\{\zeta_i\}_{i=1}^d$ denote the eigenvalues of $\Sigma_t(y)^{-1/2}\,\Sigma_t(x)\,\Sigma_t(y)^{-1/2}$.
    From Assumptions~\ref{ass:mult_noise} and~\ref{ass:mult_two}, we have $\zeta_i \ge \lambda/\Lambda$ for all $i=1,\dotsc,d$.
    Also, since $f(1) = f'(1) = 0$ and $f''(\zeta) = 1/\zeta^2 \le (\Lambda/\lambda)^2$ for all $\zeta \ge \lambda/\Lambda$, it follows that $f(\zeta_i) \le \frac{\Lambda^2}{2\lambda^2}\,{(\zeta_i-1)}^2$ for all $i=1,\dotsc,d$.
    Thus,
    \begin{align*}
        \sum_{i=1}^d f(\zeta_i)
        \le \frac{\Lambda^2}{2\lambda^2} \sum_{i=1}^d {(\zeta_i - 1)}^2
        = \frac{\Lambda^2}{2\lambda^2} \,\norm{\Sigma_t(y)^{-1/2} \,\Sigma_t(x) \,\Sigma_t(y)^{-1/2}-I}_{\rm HS}^2
        \le \frac{\Lambda^2}{2\lambda^4}\,\norm{\Sigma_t(x) - \Sigma_t(y)}_{\rm HS}^2\,.
    \end{align*}
    We can also expand
    \begin{align*}
        \norm{\Sigma_t(x) - \Sigma_t(y)}_{\rm HS}^2
        &= \norm{\sigma_t(x) \,\sigma_t(x)^\T - \sigma_t(y)\,\sigma_t(y)^\T}_{\rm HS}^2 \\
        &\le \,\bigl( \norm{\sigma_t(x) \,\sigma_t(x)^\T - \sigma_t(x)\,\sigma_t(y)^\T}_{\rm HS} + 
        \norm{\sigma_t(x) \,\sigma_t(y)^\T - \sigma_t(y)\,\sigma_t(y)^\T}_{\rm HS}\bigr)^2 \\
        &\le 2\,\bigl( \norm{\sigma_t(x) \,\sigma_t(x)^\T - \sigma_t(x)\,\sigma_t(y)^\T}_{\rm HS}^2 + 
        \norm{\sigma_t(x) \,\sigma_t(y)^\T - \sigma_t(y)\,\sigma_t(y)^\T}_{\rm HS}^2\bigr) \\
        &\le 4\Lambda \,\norm{\sigma_t(x) - \sigma_t(y)}_{\rm HS}^2 \\
        &\le 8\Lambda\,\bigl(-\alpha\,\norm{x-y}^2 - \langle b_t(x) - b_t(y), x-y\rangle\bigr)
        \le 8\,(\beta-\alpha)\,\Lambda\,\norm{x-y}^2\,.
    \end{align*}
    
    Combining everything together,
    \begin{align*}
        \KL(\delta_x P \mmid \delta_y P)
        &\le \frac{1-2\alpha h + \beta^2 h^2 + 4\,(\beta-\alpha)\, (\Lambda/\lambda)^3\,h}{2\lambda h}\,\norm{x-y}^2\,.
    \end{align*}

    \textbf{Lipschitzness of the kernel.} Next, consider a synchronous coupling of $\delta_x P$ and $\delta_y P$, i.e., we use the same noise random variable $\xi \sim \cN(0, I)$.
    Then, by Assumptions~\ref{ass:mult_noise} and~\ref{ass:mult_two},
    \begin{align*}
        W_2^2(\delta_x P, \delta_y P)
        &\le \E[\norm{x+h\,b_t(x) + \sqrt h \,\sigma_t(x)\,\xi - y - h\,b_t(y)-\sqrt h\,\sigma_t(y)\,\xi}^2] \\
        &= \norm{x+h\,b_t(x) - y-h\,b_t(y)}^2 + h \,\E[\norm{(\sigma_t(x)-\sigma_t(y))\,\xi}^2] \\
        &= \norm{x-y}^2 +2h\,\langle b_t(x)-b_t(y), x-y\rangle + \norm{b_t(x)-b_t(y)}^2 + h\,\norm{\sigma_t(x) - \sigma_t(y)}_{\rm HS}^2 \\
        &\le (1-2\alpha h + \beta^2 h^2)\,\norm{x-y}^2\,.
    \end{align*}

    \textbf{Concluding the proof.} We can now invoke Theorem~\ref{thm:discrete-main}.
    Indeed, to prove a reverse transport inequality for KL divergence, it suffices to have $W_2$-Lipschitzness of the kernel (see Remark~\ref{rmk:improving_winf}).
    Note that Theorem~\ref{thm:discrete-main} was stated for repeated applications of a single Markov kernel $P$, whereas (due to the time dependence of the coefficients of the process~\eqref{eq:discretized_ito}), here the Markov kernel changes with each iteration.
    However, since our bounds for the one-step regularity and Lipschitz constant of the kernel are uniform in time, it is easy to see that the proof of Theorem~\ref{thm:discrete-main} can be adapted to this case straightforwardly.
\end{proof}

\subsection{Sums of i.i.d.\ random variables}\label{ssec:ext:clt}

Here, we leverage the discrete-time nature of our arguments to establish a reverse transport inequality for i.i.d.\ sum processes.
Let $\rho$ be a probability density on $\R^d$ with zero mean, and let $P_h$ denote the Markov kernel representing convolution with the rescaled distribution $\rho_h(\cdot) \deq h^{-d}\, \rho(\cdot/h)$.
We will apply the shifted chain rule for the KL divergence in order to bound the quantity $\KL(\delta_x P_h^N \mmid \delta_y P_h^N)$ where $h = \frac{1}{\sqrt N}$ is chosen according to the central limit scaling and we send $N\to\infty$.

Our argument relies on the Taylor expansion of the log-density $\log \rho$, and hence we adopt the following assumptions to facilitate the proof.

\begin{ass}\label{ass:clt_setting}
    The density $\rho$ is strictly positive on $\R^d$. Also, the log-density $\log \rho$ is twice continuously differentiable and there exists $\varepsilon > 0$ such that $\int (\sup_{B(z,\varepsilon)}{\norm{\nabla^2 \log \rho}})\,\rho(\D z) <\infty$, where $B(z,\varepsilon)$ is the ball of radius $\varepsilon$ centered at $z$.
\end{ass}

In the statistics literature, the so-called \emph{differentiability in quadratic mean} (DQM) condition, which amounts to $L^2$ differentiability of the \emph{square root} of the density, has been shown to imply important consequences such as local asymptotic normality of the log-likelihood (c.f.~\cite[\S 7.2]{Vaart98Asymptotic}).
Although the DQM condition appears to be too weak to establish the following theorem, Assumption~\ref{ass:clt_setting} is certainly stronger than necessary.
We leave the problem of formulating the minimal set of assumptions for future work.

\begin{theorem}\label{thm:clt_setting}
    Let $\rho$ be a probability density on $\R^d$ with zero mean, satisfying Assumption~\ref{ass:clt_setting}.
    Let $h \deq \frac{1}{\sqrt N}$ and let $P_h$ stand for the Markov kernel representing convolution with the rescaled density $\rho_h(\cdot) \deq h^{-d}\,\rho(\cdot/h)$.
    Then, for all $x,y \in \R^d$,
    \begin{align*}
        \limsup_{N\to\infty} \KL(\delta_x P_h^N \mmid \delta_y P_h^N)
        &\le \frac{1}{2} \,\bigl\langle y-x, \bigl( \E_\rho \nabla^2 \log \frac{1}{\rho} \bigr)\,(y-x)\bigr\rangle\,.
    \end{align*}
\end{theorem}

\begin{example}[Gaussian convolution]\label{ex:clt-gaussian}
    Consider $\rho = \cN(0,\Sigma)$, where $\Sigma \succ 0$. Then $\rho_h = \cN(0,h^2 \Sigma)$, so $\delta_x P_h^N = \cN(x, Nh^2 \Sigma) = \cN(x, \Sigma)$ because $h = N^{-1/2}$, thus 
    \begin{align*}
        \KL(\delta_x P_h^N \mmid \delta_y P_h^N) = \frac{1}{2} \,\langle x-y, \Sigma^{-1}\,(x-y)\rangle\,.
    \end{align*}
    Thus in this setting of Gaussian $\rho$, Theorem~\ref{thm:clt_setting} is tight because $\E_\rho\nabla^2 \log(1/\rho) = \Sigma^{-1}$.
\end{example}

The matrix $\E_\rho \nabla^2 \log(1/\rho)$ is usually called the \emph{Fisher information matrix}, and from integration by parts it can also be written $\E_\rho[{(\nabla \log \rho)}^{\otimes 2}]$. While the Fisher information matrix is equal to the inverse covariance in the special case of Gaussians (Example~\ref{ex:clt-gaussian}), this equality does not hold in general. The Cram\'{e}r{--}Rao inequality (see~\cite[Appendix A]{ChePoo22Caffarelli} for a self-contained proof) states that $\E_\rho\nabla^2 \log(1/\rho) \succeq {(\cov_{\rho})}^{-1}$, thus the bound in Theorem~\ref{thm:clt_setting}
is always at least as big as
\begin{align}\label{eq:clt_true_ans}
    \frac{1}{2} \,\langle y-x, {(\cov_\rho)}^{-1}\,(y-x)\rangle\,,
\end{align}
which is what we would expect from CLT heuristics. In fact, under our (stringent) assumptions, the proof above actually \emph{implies} the Cram\'{e}r{--}Rao inequality.
Indeed, the central limit theorem implies $\delta_z P_h^N \to \cN(z,\cov_\rho)$ weakly for any $z\in\R^d$, and together with the lower semicontinuity of the KL divergence and Theorem~\ref{thm:clt_setting},
\begin{align*}
    \frac{1}{2}\,\langle y-x, {(\cov_\rho)}^{-1}\,(y-x)\rangle
    &= \KL(\cN(x, \cov_\rho) \mmid \cN(y,\cov_\rho)) \\
    &\le \liminf_{N\to\infty} \KL(\delta_x P_h^N \mmid \delta_y P_h^N)
    \le \frac{1}{2}\,\bigl\langle y-x, \bigl(\E_\rho \nabla^2 \log \frac{1}{\rho}\bigr)\,(y-x)\bigr\rangle\,.
\end{align*}
Since this holds for all $x,y\in\R^d$, we conclude that ${(\cov_\rho)}^{-1} \preceq \E_\rho\nabla^2 \log(1/\rho)$.

Although we have phrased Theorem~\ref{thm:clt_setting} as an asymptotic statement, non-asymptotic statements can be extracted from the proof below.
For example, if $\nabla^2 \log \rho$ is Lipschitz, then one obtains a non-asymptotic version of Theorem~\ref{thm:clt_setting} with an error term of order $O(1/\sqrt N)$.
We conjecture that under suitable assumptions, the non-asymptotic bounds can be further replaced by bounds with leading term~\eqref{eq:clt_true_ans}, but we are unable to reach this with our techniques.

\begin{proof}[Proof of Theorem~\ref{thm:clt_setting}]
    By translation invariance, it suffices to bound $\KL(\rho_h(\cdot) \mmid \rho_h(\cdot - v))$ where $v \deq (y-x)/N$.
    Taylor expansion of the log-density yields
    \begin{align}
        \KL(\rho_h(\cdot) \mmid \rho_h(\cdot - v))
        &= \int \log\Bigl(\frac{\rho(z)}{\rho(z-v/h)}\Bigr)\, \rho(\D z) \nonumber \\
        &= \int \Bigl( \bigl\langle \nabla \log \rho(z), \frac{v}{h} \bigr\rangle - \frac{1}{2}\,\bigl\langle \nabla^2 \log \rho(\tilde z), \bigl( \frac{v}{h}\bigr)^{\otimes 2}\bigr\rangle\Bigr)\,\rho(\D z)\,,\label{eq:expand_kl}
    \end{align}
    where $\tilde z$ is a point lying between $z$ and $z-v/h$.
    In particular, since $v/h = (y-x)/\sqrt N$, we have $\norm{\tilde z - z} \lesssim 1/\sqrt N$ uniformly in $z$.
    The first term in this expansion vanishes due to integration by parts.
    For the second term, we observe that pointwise,
    \begin{align*}
        N\,\bigl\langle \nabla^2 \log \rho(\tilde z), \bigl( \frac{v}{h}\bigr)^{\otimes 2}\bigr\rangle
        &= \langle \nabla^2 \log \rho(\tilde z), {(y-x)}^{\otimes 2}\rangle
        \to \langle \nabla^2 \log \rho(z), {(y-x)}^{\otimes 2}\rangle
    \end{align*}
    by continuity.
    On the other hand, for large $N$,
    \begin{align*}
        N\,\bigl\lvert \bigl\langle \nabla^2 \log \rho(\tilde z), \bigl( \frac{v}{h}\bigr)^{\otimes 2}\bigr\rangle\bigr\rvert
        &\le \bigl(\sup_{B(z,\varepsilon)}{\norm{\nabla^2 \log \rho}}\bigr)\,\norm{y-x}^2\,.
    \end{align*}
    Thus by Assumption~\ref{ass:clt_setting} we may appeal to the dominated convergence theorem, which gives
    \begin{align}\label{eq:kl_dom_conv}
        N \int \bigl\langle \nabla^2 \log \rho(\tilde z), \bigl( \frac{v}{h}\bigr)^{\otimes 2}\bigr\rangle\,\rho(\D z)
        \to \int \langle \nabla^2 \log \rho(z), {(y-x)}^{\otimes 2}\rangle \, \rho(\D z)\,.
    \end{align}

    To conclude the proof, we can appeal to the shifted chain rule. In the present setting, however, we can provide a more direct argument which could be illuminating.
    We wish to bound
    \begin{align*}
        \KL(\delta_x P_h^N \mmid \delta_y P_h^N)
        &= \KL\Bigl(\law\Bigl(x+\frac{1}{\sqrt N} \sum_{i=1}^N \xi_i\Bigr) \Bigm\Vert \law\Bigl(y + \frac{1}{\sqrt N} \sum_{i=1}^N \xi_i\Bigr)\Bigr)\,,
    \end{align*}
    where $(\xi_i : i=1,\dotsc,N)$ is a family of i.i.d.\ random variables with law $\rho$.
    By the data-processing inequality, this quantity is at most
    \begin{align*}
        \KL\Bigl(\law\Bigl(x + \frac{1}{\sqrt N} \sum_{i=1}^j \xi_i, \; j=0,1,\dotsc,N-1 \Bigr) \Bigm\Vert \law\Bigl(x + jv + \frac{1}{\sqrt N} \sum_{i=1}^j \xi_i, \; j=0,1,\dotsc,N-1 \Bigr)\Bigr)\,,
    \end{align*}
    since the last coordinates of these $N$-tuples are $x+\frac{1}{\sqrt N} \sum_{i=1}^N \xi_i$ and $y+\frac{1}{\sqrt N} \sum_{i=1}^N \xi_i$ respectively.
    Writing $\mu_j$ for the law of $x + \frac{1}{\sqrt N} \sum_{i=1}^j \xi_j$, we can now apply the chain rule for the KL divergence to bound this quantity by
    \begin{align*}
        \sum_{j=0}^{N-1} \int\KL\Bigl(\law\bigl(\zeta + \frac{1}{\sqrt N}\,\xi_{j+1}\bigr) \Bigm\Vert \law\bigl( \zeta + v + \frac{1}{\sqrt N}\,\xi_{j+1}\bigr) \Bigr) \,\mu_j(\D\zeta)
        = N\,\KL(\rho_h(\cdot) \mmid \rho_h(\cdot-v))\,.
    \end{align*}
    The claimed result now follows from~\eqref{eq:expand_kl} and~\eqref{eq:kl_dom_conv}.
\end{proof}  	
    % !TEX root = ../scp1.tex

\section{Applications to Harnack inequalities}\label{sec:apps}

In this section, we give applications of our results to Harnack inequalities. Background on Harnack inequalities is provided in \S\ref{ssec:apps:background}. Originally, these inequalities were established in~\cite{Wang1997LSINoncompact} via semigroup methods, but since they are known to be dual to reverse transport inequalities (this duality is recalled in \S\ref{ssec:apps:duality}), they also follow as a consequence of our information-theoretic methods.
We exploit this to give new Harnack inequalities in \S\ref{ssec:apps:reverse} and \S\ref{ssec:apps:discrete}.

\subsection{Background on Harnack inequalities}\label{ssec:apps:background}

In this section, we briefly provide background on Harnack inequalities.
Let a Markov semigroup ${(P_t)}_{t\ge 0}$ and a compactly supported, positive, and smooth function $f : \R^d\to\R$ be given.
A parabolic Harnack inequality, in the strictest sense of the term, might refer to an inequality of the form
\begin{align}\label{eq:pure_harnack}
    P_t f(x) \le C(x,y,t)\, P_t f(y)\qquad\text{for all}~x,y\in\R^d\,.
\end{align}
While Harnack inequalities have played a central role in the theory of elliptic PDEs, parabolic Harnack inequalities turn out to be more subtle.
Namely, it is well-known that for typical diffusion processes (such as the Langevin diffusion~\eqref{eq:langevin}), an inequality such as~\eqref{eq:pure_harnack} cannot hold for all functions $f$.
To circumvent this, in the pioneering work~\cite{LiYau1986Parabolic}, P.\ Li and S.-T.\ Yau introduced an alternative Harnack inequality which instead compares the semigroup at two different times $t$ and $s+t$.
Their inequality, however, depends on the ambient dimension, which is at odds with the intrinsically infinite-dimensional character of many diffusion processes (in the sense of Bakry{--}\'{E}mery).
Hence, in~\cite{Wang1997LSINoncompact}, F.-Y.\ Wang introduced an infinite-dimensional Harnack inequality, obtained through commutation of the power function ${(\cdot)}^p$ ($p > 1$) with the
semigroup; we refer to these inequalities as \emph{power Harnack inequalities}:
\begin{align}\label{eq:power_harnack}
    {(P_t f(x))}^p
    &\le P_t(f^p)(y)\exp\Bigl(\frac{\alpha p\,\norm{x-y}^2}{2\,(p-1)\,(\exp(2\alpha t)-1)}\Bigr)\,, \qquad \forall \,x,y\in\R^d, \; t > 0\,.
\end{align}
By replacing $f$ with $f^{1/p}$ and letting $p \to \infty$ in~\eqref{eq:power_harnack}, one obtains the \emph{log Harnack inequality}
\begin{align}\label{eq:log_harnack}
    P_t(\log f)(x) \le \log P_t f(y) + \frac{\alpha\,\norm{x-y}^2}{2\,(\exp(2\alpha t) - 1)}\,.
\end{align}

\begin{table}
    \centering
    \begin{tabular}{ccc}
        \textbf{Name} & \textbf{Inequality} & \textbf{Notes} \\
         Pure Harnack & $P_t f(x) \le C(x,y,t)\, P_t f(y)$ & Usually fails to hold \\
         Power Harnack & $P_t f(x) \le C_p(x,y,t)\, {P_t(f^p)(y)}^{1/p}$ & $p \in (1,\infty)$ \\
         Log Harnack & $P_t f(x) \le C(x,y,t) +\log P_t(\exp f)(y)$ & \\
         Power Harnack & $P_t f(x) \ge C_p(x,y,t)\,{P_t(f^p)(y)}^{1/p}$ & $p \in (0,1)$ \\
         Reverse Harnack & $P_t f(x) \ge C_p(x,y,t)\,{P_t(f^p)(y)}^{1/p}$ & $p \in (-\infty, 0)$ (see \S\ref{ssec:apps:reverse})
    \end{tabular}
    \caption{\footnotesize{A family of parabolic Harnack inequalities.}}
    \label{tab:harnack}
\end{table}

In Table~\ref{tab:harnack}, we record these inequalities in a form which emphasizes their similarities. This table includes power Harnack inequalities for $p \in (0, 1)$, which are obviously equivalent to the $p \in (1,\infty)$ case up to replacing $p$ with $1/p$, as noted in~\cite{AnZon22BilateralHarnack}.
We also include \emph{reverse Harnack inequalities} corresponding to the case $p \in (-\infty, 0)$, named in analogy to the family of reverse hypercontractivity inequalities, which we explore in \S\ref{ssec:apps:reverse}.

\paragraph*{Equivalences with the curvature-dimension condition.}
Part of the importance of Harnack inequalities stems from their equivalence to a large family of other properties, including the curvature-dimension condition.
We list a few of these equivalences below, focusing on the Langevin semigroup for ease of exposition although the equivalences hold much more generally.
The reader can find further equivalences in the book~\cite{bakry2014analysis} or in~\cite[Theorem 2.3.3]{Wang14Diffusion}.

\begin{theorem}\label{thm:equiv_cd}
    Let ${(P_t)}_{t\ge 0}$ denote the Markov semigroup corresponding to the Langevin diffusion~\eqref{eq:langevin} with potential $V$, and let $\alpha\in\R$, $p, q > 1$. The following are equivalent.
    \begin{enumerate}
        \item (Curvature-dimension condition) $\nabla^2 V \succeq \alpha I$ on $\R^d$.
        \item (Wasserstein contraction) For any $t > 0$ and any $x,y\in \R^d$,
        \begin{align*}
            W_p(\delta_x P_t, \delta_y P_t) \le \exp(-\alpha t)\,\norm{x-y}\,.
        \end{align*}
        \item (Gradient bound) For any $f \in \mc C_{\rm c}^\infty(\R^d)$ and any $t > 0$,
        \begin{align*}
            \norm{\nabla P_t f} \le \exp(-\alpha t)\,P_t\norm{\nabla f}\,.
        \end{align*}
        \item (Power Harnack) The inequality~\eqref{eq:power_harnack} holds.
        \item (R\'enyi regularity) For any $t > 0$ and any $x, y \in \R^d$,
        \begin{align}\label{eq:rev_ren_transport}
            \Ren_q(\delta_x P_t \mmid \delta_y P_t) \le \frac{\alpha q\,\norm{x-y}^2}{2\,(\exp(2\alpha t) - 1)}\,.
        \end{align}
        \item (Log Harnack) The inequality~\eqref{eq:log_harnack} holds.
        \item (KL regularity) For any $t > 0$ and any $x,y\in\R^d$,
        \begin{align}\label{eq:rev_kl_transport}
            \KL(\delta_x P_t \mmid \delta_y P_t) \le \frac{\alpha\,\norm{x-y}^2}{2\,(\exp(2\alpha t) - 1)}\,.
        \end{align}
    \end{enumerate}
\end{theorem}

For example,~\cite{Wang10HarnackBoundary} showed that the validity of~\eqref{eq:power_harnack} for any fixed $p > 1$ implies the log-Harnack inequality, which in turn implies the curvature-dimension condition. We give a dual version of the former statement, namely that~\eqref{eq:rev_ren_transport} implies~\eqref{eq:rev_kl_transport}, in \S\ref{app:renyi_implies_kl}.

\paragraph*{Implications of the Harnack inequalities.} The dimension-free Harnack inequalities encode a wealth of information about the semigroup and have been successfully applied to establish functional inequalities~\cite{Wang1997LSINoncompact, Wang1999HarnackLSI, bobgenled2001hypercontractivity, Wang01LSIConditions, Wang17HyperHam}, heat kernel estimates~\cite{AidKaw01ShortTime, GonWan01HeatKernel, AidZha02SmallTimeGroups}, higher-order eigenvalue estimates~\cite{Wang02SpectrumInfinite, GonWan04GromovHodge}, and ultracontractivity~\cite{Wang06HarnackApplications}; see, e.g.,~\cite[\S 1]{Wang14Diffusion} for further statements.

Let us briefly note that these properties correspond to regularity of Kolmogorov's \emph{backward} equation.
Indeed, we will show that the reverse transport inequality~\eqref{eq:rev_ren_transport} for $q=2$ yields
\begin{align}\label{eq:local_poincare}
    P_t(f^2) - {(P_t f)}^2 \ge \frac{\exp(2\alpha t) - 1}{\alpha}\,\norm{\nabla P_t f}^2\,, \qquad \forall\,t > 0\,.
\end{align}
Conversely,~\eqref{eq:local_poincare} is a form of the \emph{local Poincar\'e inequality} which is equivalent to the curvature-dimension condition (c.f.~\cite[Theorem 4.7.2]{bakry2014analysis}), and therefore implies back~\eqref{eq:rev_ren_transport} for any $q > 1$ by Theorem~\ref{thm:equiv_cd}.

The inequality~\eqref{eq:local_poincare} implies that the semigroup $P_t$ maps bounded measurable functions to differentiable functions, which is a smoothing property of the semigroup.
On the other hand, in~\cite{Wang14ShiftHarnack}, F.-Y.\ Wang introduced the family of \emph{shift Harnack inequalities} which instead capture the regularity of Kolmogorov's forward equation, and we will revisit them from the lens of information theory in a forthcoming work.

\begin{proof}[Proof of~\eqref{eq:local_poincare} from~\eqref{eq:rev_ren_transport} for $q=2$]
    For $h\in\R^d \setminus \{0\}$, by the Cauchy{--}Schwarz inequality,
    \begin{align*}
        \Bigl\lvert \frac{P_t f(x+h) - P_t f(x)}{\norm h}\Bigr\rvert
        &= \frac{1}{\norm h} \,\Bigl\lvert\int f\,\bigl( \frac{\D (\delta_{x+h} P_t)}{\D (\delta_x P_t)} - 1\bigr) \, \D (\delta_x P_t)\Bigr\rvert \\
        &\le \sqrt{\var_{\delta_x P_t}(f)}\, \frac{\sqrt{\chi^2(\delta_{x+h} P_t \mmid \delta_x P_t)}}{\norm h}\,.
    \end{align*}
    Applying~\eqref{eq:rev_ren_transport} and sending $\norm h\searrow 0$,
    \begin{align*}
        \norm{\nabla P_t f(x)}
        &\le \sqrt{\var_{\delta_x P_t}(f)\, \frac{\alpha}{\exp(2\alpha T) - 1}}\,.
    \end{align*}
    Square both sides of the inequality to recover the result.
\end{proof}

\subsection{Duality between Harnack and reverse transport inequalities}\label{ssec:apps:duality}

Since the equivalences in Theorem~\ref{thm:equiv_cd} are by now well-known, we will not prove them all in this paper. However, since we will later use the duality between reverse transport inequalities and Harnack inequalities to deduce the latter from the former, we recall this duality below.

\begin{proof}[Equivalence between~\eqref{eq:power_harnack} and~\eqref{eq:rev_ren_transport} for $q=\frac{p}{p-1}$, and between~\eqref{eq:log_harnack} and~\eqref{eq:rev_kl_transport}.]
    Let $P$ be a Markov kernel on a Polish space $\cX$ and write $C_p(x,y)$ for the best constant in the power Harnack inequality
    \begin{align*}
        P f(x) \le C_p(x,y)\,{P(f^p)(y)}^{1/p}\,.
    \end{align*}
   The best constant $C_p(x,y)$ is simply the operator norm $\norm{L_x}_{L^p(\delta_y P) \to \R}$ of the linear function $L_x : f \mapsto P f(x)$ because ${P(f^p)(y)}^{1/p} = \norm f_{L^p(\delta_y P)}$.
    By re-writing $P f(x) = \int f \, \D(\delta_x P) = \int f \,\frac{\D(\delta_x P)}{\D(\delta_y P)} \, \D (\delta_y P)$ and appealing to H\"older duality, this operator norm equals
    \begin{align*}
        C_p(x,y)
        =
        \norm{L_x}_{L^p(\delta_y P) \to \R}
        &= \Bigl\lVert \frac{\D(\delta_x P)}{\D(\delta_y P)} \Bigr\rVert_{L^q(\delta_y P)} = \exp\Bigl(\frac{q-1}{q}\,\Ren_q(\delta_x P \mmid \delta_y P)\Bigr)
    \end{align*}
    which yields the equivalence between~\eqref{eq:power_harnack} and~\eqref{eq:rev_ren_transport}.
    In particular, $\Ren_q(\delta_x P \mmid \delta_y P) \le \rho(x,y)$ if and only if $C_p(x,y) \le \exp(\frac{q-1}{q}\,\rho(x,y))$.
    \par The equivalence between~\eqref{eq:log_harnack} and~\eqref{eq:rev_kl_transport} follows along similar lines, replacing H\"older duality with the Donsker{--}Varadhan variational principle
    \begin{align*}
        \int f \, \D \mu \le \KL(\mu \mmid \nu) + \log \int \exp(f) \, \D \nu\,,
    \end{align*}
    with equality if and only if $f = \log \frac{\D\mu}{\D\nu}$.
    It yields that if $C_{\log}(x,y)$ is the best constant in the inequality
    \begin{align*}
        Pf(x) \le C_{\log}(x,y) + \log P(\exp f)(y)\,,
    \end{align*}
    then $\KL(\delta_x P \mmid\delta_y P) \le \rho(x,y)$ if and only if $C_{\log}(x,y) \le \rho(x,y)$.
\end{proof}

\begin{remark}[Distributional Harnack inequalities]\label{rem:distributional-harnack}
This duality argument extends from measures $\delta_x P$, $\delta_y P$ to measures $\mu P$, $\nu P$. Through the convexity principle in \S\ref{ssec:disc:convexity}, we have obtained reverse transport inequalities from arbitrary initializations $\mu$, $\nu$.
By dualizing these results, we obtain \emph{distributional} Harnack inequalities: under $\CD(\alpha,\infty)$, it holds that for $p > 1$, $t > 0$, and $x,y\in\R^d$,
\begin{align}\label{eq:dist_power_harnack}
    \norm f_{L^1(\mu P_t)}
    &\le \norm f_{L^p(\nu P_t)} \inf_{\gamma \in \Coup(\mu,\nu)} 
 \Bigl[\int \exp\Bigl(\frac{\alpha p\,\norm{x-y}^2}{2\,{(p-1)}^2\,(\exp(2\alpha t) - 1)}\Bigr)\,\gamma(\D x, \D y)\Bigr]^{(p-1)/p}
\end{align}
and
\begin{align}\label{eq:dist_log_harnack}
    \E_{\mu P_t} f
    &\le \log \E_{\nu P_t}\exp(f) + \frac{\alpha\,W_2^2(\mu,\nu)}{2\,(\exp(2\alpha t) - 1)}\,.
\end{align}
We show how to obtain these distributional Harnack inequalities via direct arguments (i.e., without dualizing standard Harnack inequalities, appealing to the convexity principle, and dualizing back), and similarly for the dual versions of the refined R\'enyi bounds of Theorem~\ref{thm:renyi_cvxty_principle_refined}, in \S\ref{app:dist_harnack}.
\end{remark}

\subsection{Reverse Harnack inequalities}\label{ssec:apps:reverse}

As we have shown in the preceding section, the family of reverse transport inequalities of order $q \in (1,\infty)$ is equivalent, via duality, to the power Harnack inequalities with exponent $p \in (1,\infty)$, with an additional equivalence between the case of $q= 1$ (i.e., the KL divergence) and the log-Harnack inequality.
In \S\ref{sec:disc}, we also established reverse transport inequalities of order $q \in (0, 1)$, for which the corresponding dual exponent $p \deq -q/(1-q)$ ranges in $(-\infty, 0)$.
In this section, we will show that we consequently obtain a family of \emph{reverse} Harnack inequalities.

The following lemma provides the key duality result.

\begin{lemma}[Duality]\label{lem:reverse-harnack-helper}
    Let $q \in (0,1)$ and $p = -\frac{q}{1-q} \in (-\infty,0)$. For any positive function $f$ and any probability measures $\mu$, $\nu$,
    \[
        \E_{\mu} f \geq (\E_{\nu} f^p)^{1/p} \exp\Bigl( -\frac{1}{\abs p}\, \Ren_q(\mu \mmid \nu) \Bigr)\,.
    \]
\end{lemma}
\begin{proof}
    Let $g \deq f^q$. To simplify the notation, we assume that $\mu \ll \nu$, although this assumption is not necessary for the proof. Then
    \begin{align*}
        \exp((q-1)\, \Ren_q(\mu \mmid \nu))
        = \E_{\nu}\bigl[ \frac{g}{g}\, \bigl( \frac{\D\mu}{\D\nu} \bigr)^q \bigr]
        \leq \E_{\mu} [ g^{1/q}]^q \, \E_{\nu} [g^{-1/(1-q)}]^{1-q}
        = \E_{\mu} [f]^q \, \E_{\nu} [f^{-q/(1-q)}]^{1-q}\,.
    \end{align*}
    Above, the first step is the definition of R\'enyi divergence and multiplying and dividing by the same quantity $g$,
    the second step is by H\"older's inequality with dual exponents $1/q$ and $1/(1-q)$, and the final step is by definition of $g$. Rearranging the above display, raising everything to the power of $1/q$, and simplifying by using the definition of $p$ completes the proof.
\end{proof}

From Theorem~\ref{thm:langevin-disc} and Lemma~\ref{lem:reverse-harnack-helper}, we immediately obtain the following inequalities.

\begin{theorem}[Reverse Harnack inequalities]
    Let ${(P_t)}_{t\ge 0}$ denote the Langevin semigroup corresponding to a potential satisfying $\nabla^2 V \succeq \alpha I$ on $\R^d$, where $\alpha\in\R$.
    Then, for all functions $f : \R^d\to\R_{>0}$, all $p \in (-\infty, 0)$, all $t > 0$, and all $x,y\in\R^d$, it holds that
    \begin{align}\label{eq:reverse_harnack}
        P_t f(x)
        &\ge {P_t(f^p)(y)}^{1/p} \exp\Bigl( - \frac{\alpha\,\norm{x-y}^2}{2\,\abs{p-1}\,(\exp(2\alpha t) - 1)}\Bigr)\,.
    \end{align}
\end{theorem}

We remark that replacing $f$ with $f^{1/p}$ transforms the reverse Harnack inequality with exponent $p$ into the one with exponent $1/p$.
This fact corresponds to the relation $\Ren_q(\mu \mmid \nu) = \frac{q}{1-q}\,\Ren_{1-q}(\nu \mmid \mu)$ which holds for $q \in (0,1)$.

Once the form of the reverse Harnack inequalities are known, it is not difficult to prove them. For example, we show in \S\ref{app:semigp_reverse} that for a diffusion process on a Riemannian manifold for which the curvature-dimension condition $\CD(\alpha,\infty)$ holds, they can be established via the usual semigroup calculations.
In fact, they are implied by the power Harnack inequalities~\eqref{eq:power_harnack} through a simple application of Jensen's inequality: for $p > 1$, replacing $f$ with $f^{1/p}$ in~\eqref{eq:power_harnack} and interchanging $x$ and $y$, followed by applying Jensen's inequality to the convex function ${(\cdot)}^{-1}$, yields
\begin{align*}
    P_t f(x)
    &\ge {P_t(f^{1/p})(y)}^p \exp\Bigl( -\frac{\alpha p\,\norm{x-y}^2}{2\,(p-1)\,(\exp(2\alpha t) -1)}\Bigr) \\
    &\ge {P_t(f^{-1/p})(y)}^{-p} \exp\Bigl( -\frac{\alpha p\,\norm{x-y}^2}{2\,(p-1)\,(\exp(2\alpha t) -1)}\Bigr)\,,
\end{align*}
which is seen to be equivalent to~\eqref{eq:reverse_harnack} for exponent $-1/p \in (-\infty, 0)$.

On the other hand, we also show in \S\ref{app:semigp_reverse} that the reverse Harnack inequality~\eqref{eq:reverse_harnack} implies back $\CD(\alpha,\infty)$. We can therefore add two more equivalences to Theorem~\ref{thm:equiv_cd}.

\begin{theorem}
    Consider the setting of Theorem~\ref{thm:equiv_cd}, let $p \in (-\infty, 0)$, and let $q \in (0,1)$.
    Then, any of the statements of Theorem~\ref{thm:equiv_cd} are equivalent to any of the following.
    \begin{enumerate}\setcounter{enumi}{7}
        \item (Reverse Harnack) The inequality~\eqref{eq:reverse_harnack} holds.
        \item (R\'enyi regularity for $q \in (0,1)$) For any $t > 0$ and any $x,y\in\R^d$,
        \begin{align*}
            \Ren_q(\delta_x P_t \mmid \delta_y P_t)
            &\le \frac{\alpha q\,\norm{x-y}^2}{2\,(\exp(2\alpha t) -1)}\,.
        \end{align*}
    \end{enumerate}
\end{theorem}

\subsection{Harnack inequalities for discretizations of diffusions}\label{ssec:apps:discrete}

We conclude by noting that the techniques of \S\ref{sec:disc}, which apply to discrete-time processes, combined with the duality arguments of \S\ref{ssec:apps:duality}, enable us to prove Harnack inequalities for iterations of discrete-time Markov kernels.
Such discrete-time processes arise naturally in the study of sampling algorithms based on time-discretizations of SDEs; processes which have limiting SDE descriptions (e.g., the CLT setting of \S\ref{ssec:ext:clt}); and more generally, discrete dynamical systems (e.g.,~\cite{DjeGuiWu04TransportInfo}).
Here, we provide a simple illustration by dualizing the result of \S\ref{ssec:ext:ito}. Similar results can also be deduced in analogous way for, e.g., the settings of \S\ref{ssec:disc:langevin} and \S\ref{ssec:ext:clt}.

\begin{cor}[Log-Harnack inequality]
    Consider the setting of Theorem~\ref{thm:mult_noise} pertaining to the discrete-time process~\eqref{eq:discretized_ito}.
    Let $\hat P_h$ denote the corresponding Markov transition kernel.
    Then, for all positive measurable functions $f : \R^d\to\R$ and all $x,y\in\R^d$,
    \begin{align*}
        \hat P_h^N f(x)
        &\le \Bigl(1+\frac{4\,(\beta-\alpha)\,{(\Lambda/\lambda)}^3\,h}{L^2}\Bigr)\,\frac{\alpha\,(1-\beta h/2)}{\lambda\,(L^{-2N}-1)} + \log \hat P_h^N(\exp f)(y)\,,
    \end{align*}
    where $L^2 \deq 1-2\alpha h + \beta^2 h^2$.
\end{cor}

To our knowledge, such Harnack inequalities for discrete-time processes have not previously appeared in the literature.

	\paragraph*{Acknowledgements.} 
	We thank Guy Bresler, Yanjun Han, Jonathan Niles-Weed, Pablo Parrilo, Gabriel Peyr\'e, Yury Polyanskiy, Philippe Rigollet, Martin Wainwright, and Andre Wibisono for stimulating conversations. JMA acknowledges the support of an NYU Faculty Fellowship. SC acknowledges the support of the Eric and Wendy Schmidt Fund at the Institute for Advanced Study. 

  \newpage

	\appendix
	% !TEX root = ../scp1.tex

\section{Deferred details}\label{app:deferred}

\subsection{Proof of the R\'enyi composition rule}\label{app:pf_composition}

For completeness, here we prove the R\'enyi composition rule in Theorem~\ref{thm:renyi_prop}. 
    First we note that the case $q=1$ follows from the KL chain rule, and the case $q = \infty$ follows from the case $q < \infty$ using monotonicity of the R\'enyi divergences in the order and by taking limits.
    Therefore, we focus on the cases $q \in (0,1)$ and $q \in (1,\infty)$.

    \textbf{\underline{Case $q > 1$.}} We may assume that $\bs\mu^X \ll \bs \nu^X$ and that $\bs\mu^{Y\mid X=x} \ll \bs \nu^{Y\mid X=x}$ for $\bs\mu^X$-a.e.\ $x\in\Omega$, since otherwise the inequality we wish to prove is trivial.
    Disintegration of measure yields
    \begin{align*}
        1+\Hell_q(\bs\mu^{X,Y} \mmid \bs \nu^{X,Y})
        &= \int \Bigl(\frac{\D\bs\mu^X}{\D\bs\nu^X}(x)\Bigr)^{q-1} \,\Bigl[\int \Bigl( \frac{\D\bs\mu^{Y\mid X=x}}{\D\bs\nu^{Y\mid X=x}}(y) \Bigr)^{q-1}\,\bs \mu^{Y\mid X=x}(\D y)\Bigr] \, \bs \mu^X(\D x) \\
        &= \int \Bigl(\frac{\D\bs\mu^X}{\D\bs\nu^X}(x)\Bigr)^{q-1} \,\bigl(1+\Hell_q(\bs\mu^{Y\mid X=x} \mmid \bs \nu^{Y\mid X=x})\bigr) \, \bs \mu^X(\D x) \\
        &\le \Bigl[\esssup{\bs\mu^X}_{x\in \Omega}{\bigl(1+\Hell_q(\bs\mu^{Y\mid X=x} \mmid \bs \nu^{Y\mid X=x})\bigr)}\Bigr]\, \bigl(1+\Hell_q(\bs\mu^X \mmid \bs \mu^Y)\bigr)\,.
    \end{align*}
    The result follows by applying the increasing function $\frac{1}{q-1} \log(\cdot)$ to both sides of the above inequality.

    \textbf{\underline{Case $q < 1$.}} Here, we must be more cautious as $\Ren_q(\mu\mmid \nu) < \infty$ no longer implies $\mu \ll \nu$.
    Consider the dominating measure $\bs\lambda^{X,Y} \deq \frac{1}{2}\,(\bs\mu^{X,Y}+\bs\nu^{X,Y})$, which admits the disintegration
    \begin{align*}
        \bs\lambda^{X,Y}(\D x, \D y)
        &= \underbrace{\frac{1}{2}\,(\bs\mu^X + \bs \nu^X)(\D x)}_{\eqqcolon \bs \lambda^X(\D x)}\,\underbrace{\Bigl[\frac{\D \bs\mu^X}{\D(\bs\mu^X+\bs\nu^X)}(x)\,\bs \mu^{Y\mid X=x}(\D y) + \frac{\D\bs\nu^X}{\D(\bs\mu^X + \bs \nu^X)}(x)\,\bs \nu^{Y\mid X=x}(\D y)\Bigr]}_{\eqqcolon \bs\lambda^{Y\mid X=x}(\D y)}\,.
    \end{align*}
    Also, let
    \begin{align*}
        \Omega' \deq \Bigl\{ \frac{\D\bs\mu^X}{\D\bs\lambda^X} \wedge \frac{\D\bs\nu^X}{\D\bs\lambda^X} > 0\Bigr\}\,.
    \end{align*}
    From these expressions, $\bs\mu^{Y\mid X=x} \ll \bs\lambda^{Y\mid X=x}$ and $\bs\nu^{Y\mid X=x} \ll \bs\lambda^{Y\mid X=x}$ for $\bs\lambda^X$-a.e.\ $x\in\Omega'$.
    Therefore,
    \begin{align*}
        &1-\Hell_q(\bs \mu^{X,Y} \mmid \bs \nu^{X,Y})
        = \int_{\Omega'} \Bigl( \frac{\D \bs \mu^{X,Y}}{\D \bs\lambda^{X,Y}}\Bigr)^q \, \Bigl( \frac{\D\bs\nu^{X,Y}}{\D\bs\lambda^{X,Y}}\Bigr)^{1-q}\, \D \bs \lambda^{X,Y} \\
        &\qquad = \int_{\Omega'} \Bigl( \frac{\D \bs \mu^{X}}{\D \bs\lambda^{X}}(x)\Bigr)^q \, \Bigl( \frac{\D\bs\nu^{X}}{\D\bs\lambda^{X}}(x)\Bigr)^{1-q}\,\Bigl[\int \Bigl( \frac{\D \bs \mu^{Y\mid X=x}}{\D \bs\lambda^{Y\mid X=x}}(y)\Bigr)^q \, \Bigl( \frac{\D\bs\nu^{Y\mid X=x}}{\D\bs\lambda^{Y\mid X=x}}(y)\Bigr)^{1-q}\,\bs \lambda^{Y\mid X=x}(\D y)\Bigr]\, \bs \lambda^{X}(\D x) \\
        &\qquad = \int_{\Omega'} \Bigl( \frac{\D \bs \mu^{X}}{\D \bs\lambda^{X}}(x)\Bigr)^q \, \Bigl( \frac{\D\bs\nu^{X}}{\D\bs\lambda^{X}}(x)\Bigr)^{1-q}\,\bigl(1-\Hell_q(\bs\mu^{Y\mid X=x} \mmid \bs \nu^{Y\mid X=x})\bigr)\, \bs \lambda^{X}(\D x) \\
        &\qquad \ge \Bigl[\essinf{\bs\lambda^X}_{x\in\Omega'}{\bigl(1-\Hell_q(\bs\mu^{Y\mid X=x} \mmid \bs \nu^{Y\mid X=x})\bigr)}\Bigr] \,\bigl(1-\Hell_q(\bs \mu^X \mmid \bs \nu^X)\bigr)\,.
    \end{align*}
    Note that by definition, $\bs\lambda^X|_{\Omega'} \ll \bs\mu^X \wedge \bs \nu^X$, hence we can replace the $\bs\lambda^X$-essential infimum with the $(\bs\mu^X\wedge \bs\nu^X)$-essential infimum.
    The proof is concluded by applying the \emph{decreasing} function $\frac{1}{q-1} \log(\cdot)$ to both sides of this inequality.

\subsection{Optimizing the shifts}\label{app:calc}

\subsubsection{Discrete-time argument via single-variable calculus}\label{app:calc-discrete}

Here we explicitly provide the unique optimal solution of the shift optimization problem in the proof of Theorem~\ref{thm:discrete-main},
recalled here for convenience:
\begin{align}
    V_{N}  \deq  \min_{\substack{\eta_0, \dots, \eta_{N-1}\ge 0 \\ \text{s.t. }\eta_{N-1}=1}} \sum_{n=0}^{N-1} L^{2n} \eta_n^2 \prod_{k=0}^{n-1} (1 - \eta_k)^2\,.
    \label{eq:disc-opt:app}
\end{align}
Below, denote $R_k  \deq  \frac{L^{-2}-1}{L^{-2\,(k+1)}-1} $.

\begin{lemma}\label{lem:shift-opt:discrete}
    For any horizon $N \in \N$ and any Lipschitz constant $L > 0$, the unique optimal solution to~\eqref{eq:disc-opt:app} is $\eta_i = R_{N-1-i}$ for all $i \in \{0,\dots,N-1\}$. The corresponding optimal value is $V_N = R_{N-1}$.
\end{lemma}

\begin{proof}
    We prove by induction on $N$. The base case $N=1$ is trivial. For the induction, observe that
    \begin{align*}
        V_{N+1} = 
        \eta_0^2 + L^2 V_N\, (1 - \eta_0)^2\,,
    \end{align*}
    by using the unique optimal $\eta_0, \dots, \eta_{N-1}$ for the problem of horizon $N$ as the respective solutions for $\eta_1, \dots, \eta_N$ for the problem of horizon $N+1$. It remains to solve for $\eta_0$. This follows immediately from the simple observation stated below and the identity $R_N = \frac{1}{1+{(L^2R_{N-1})}^{-1}}$.
\end{proof}

 \begin{obs}
    For any $a \geq 0$, the optimization problem
    \[  
        \min_{\eta\ge 0}{\{\eta^2 + a\, (1 - \eta)^2\}}
    \]
    has optimal value $\frac{1}{1+a^{-1}}$, achieved at the unique optimal solution $\eta = \frac{1}{1+a^{-1}}$.
\end{obs}
\begin{proof}
    This is a straightforward calculation using single-variable calculus.
\end{proof}

\subsubsection{Continuous-time argument via calculus of variations}\label{app:calc-cont}

Here, we provide the calculus of variations derivation of the process $\{\eta_t\}_{t\in [0,T]}$ used in \S\ref{sec:cont}.
Recall that we wish to minimize the functional $\ms F$ defined by
\begin{align*}
    \ms F(\eta)
    &\deq \int_0^T \eta_t^2 \exp\Bigl(-2\alpha t - 2\int_0^t \eta_s \, \D s\Bigr)\,\D t\,.
\end{align*}
Recall also that the first variation $\delta\ms F(\eta) : [0,T]\to\R$ satisfies, by definition,
\begin{align*}
    \lim_{\varepsilon\searrow 0} \frac{\ms F(\eta+\varepsilon\chi) - \ms F(\eta)}{\varepsilon}
    &= \int_0^T \delta \ms F(\eta) \, \chi\,, \qquad\text{for every perturbation}~\chi : [0,T]\to\R\,.
\end{align*}
It suffices to assume $\alpha \ne 0$.
Elementary calculus yields
\begin{align*}
	\delta\ms F(\eta)(t)
	&= 2\eta_t \exp\Bigl(-2\int_0^t (\alpha+\eta_s)\,\D s\Bigr)
	- 2\int_t^T \eta_s^2 \exp\Bigl(-2\int_0^s (\alpha+\eta_r)\,\D r\Bigr) \, \D s\,.
\end{align*}
Setting this to zero and differentiating, we obtain the differential equation
\begin{align*}
    \dot\eta_t - 2\alpha\eta_t - \eta_t^2 = 0\,.
\end{align*}
If we set $\theta_t \deq \eta_t \exp(-2\alpha t)$, then $\dot\theta_t = (\dot\eta_t - 2\alpha \eta_t)\exp(-2\alpha t) = \eta_t^2 \exp(-2\alpha t)$, or $\dot\theta_t = \theta_t^2 \exp(2\alpha t)$.
Integration yields
\begin{align*}
    \frac{1}{\theta_0} - \frac{1}{\theta_t} =
    \int_0^t \frac{\dot \theta_s}{\theta_s^2}\, ds
    = \int_0^t \exp(2\alpha s) \, \D s
	= \frac{\exp(2\alpha t) - 1}{2\alpha}\,.
\end{align*}
Therefore,
\begin{align*}
    \eta_t
    &= \frac{2\alpha \theta_0 \exp(2\alpha t)}{2\alpha - \theta_0\,(\exp(2\alpha t) - 1)}\,.
\end{align*}
Since we require $\eta_t \nearrow \infty$ as $t\nearrow T$, we set $\theta_0 = 2\alpha/(\exp(2\alpha T) - 1)$, i.e.,
\begin{align*}
	\eta_t
	&= \frac{2\alpha}{\exp(2\alpha\,(T-t))-1}\,.
\end{align*}
Let $H_t \deq \int_0^t \eta_s \, \D s = \log \frac{1-\exp(-2\alpha T)}{1-\exp(-2\alpha\,(T-t))}$.
Then,
\begin{align*}
	\int_0^T \eta_t^2 \exp(-2\alpha t - 2H_t) \, \D t
	&= {(2\alpha)}^2 \int_0^T \frac{\exp(-2\alpha t)}{{\left(\exp(2\alpha\,(T-t))-1\right)}^2}\, \Bigl(\frac{1-\exp(-2\alpha\,(T-t))}{1-\exp(-2\alpha T)} \Bigr){\Bigsp}^2 \, \D t \\
	&= \Bigl( \frac{2\alpha}{\exp(2\alpha T) - 1} \Bigr){\Bigsp}^2 \int_0^T\exp(2\alpha t) \, \D t
	= \frac{2\alpha}{\exp(2\alpha T) - 1}\,.
\end{align*}

\subsection{Tightness}\label{app:tightness}

Here we show tightness of the R\'enyi regularity bounds and associated finiteness thresholds in \S\ref{sec:disc} and \S\ref{sec:cont} by explicitly computing these quantities for the semigroup ${(P_t)}_{t\ge 0}$ corresponding to the Ornstein{--}Uhlenbeck (OU) process
\begin{align}
    \D X_t = -\alpha X_t \,\D t + \sqrt{2}\, \D B_t\,\label{eq:tightness-ou}
\end{align}
and the associated discrete-time Markov kernel $\hat P_h$, defined as
\begin{align}
    \hat{P}_h(x,\cdot) = Q_{2h}((1 - \alpha h)\,x, \cdot)\,,\label{eq:tightness-ou-discrete}
\end{align}
where ${(Q_{t})}_{t\ge 0}$ is the heat semigroup. For brevity, we compute the relevant quantities only
for curvature-dimension $\alpha \neq 0$ (since the case $\alpha = 0$ follows by a limiting argument or by directly redoing the calculation in a completely analogous way), and for Dirac initializations (since this clearly implies tightness for general initializations).

\subsubsection{Regularity}

\paragraph*{Tightness of the discrete-time R\'enyi regularity.} This calculation closely follows the lower bound for the discrete mixing time of discretized Langevin in~\cite{AltTal23Langevin}. Consider the Markov kernel $\hat{P}_h$ in~\eqref{eq:tightness-ou-discrete} and denote $L  \deq  1 - \alpha h$ for shorthand. An explicit computation gives 
\[
    \delta_x \hat{P}_h^N = \cN\Bigl( L^N x,\, 2h\,\frac{1 - L^{2N}}{1-L^2} \Bigr)
    \qquad
    \text{and} 
    \qquad
    \delta_y \hat{P}_h^N = \cN\Bigl( L^N y, 2h\,\frac{1 - L^{2N}}{1-L^2} \Bigr)\,.
\]
Thus by the identity for the R\'enyi divergence between Gaussians (Theorem~\ref{thm:renyi_prop}), 
\begin{align*}
    \Ren_q( \delta_x \hat{P}_h^N \mmid \delta_y \hat{P}_h^N)
	=
    \frac{q\,(1-L^2)}{4h\,(L^{-2N} - 1)}\, \|x-y\|^2\,.
\end{align*}
This exactly matches the discrete-time R\'enyi regularity bound in Theorem~\ref{thm:langevin-disc}.

\paragraph*{Tightness of the continuous-time R\'enyi regularity.} While this calculation follows from the limit $h \to 0$ of the discrete-time tightness (above), we prefer to compute the relevant quantities directly in continuous time. Let $P_T$ be as defined in~\eqref{eq:tightness-ou}. An explicit computation gives
\[
    \delta_x P_T = \cN \Bigl( \exp(-\alpha T)\, x, \,\frac{1 - \exp(-2\alpha T)}{\alpha}\Bigr)
    \qquad
    \text{and}
    \qquad
    \delta_y P_T = \cN \Bigl( \exp(-\alpha T)\, y,\, \frac{1 - \exp(-2\alpha T)}{\alpha}\Bigr)
    \,.
\]
Thus by the identity for the R\'enyi divergence between Gaussians (Theorem~\ref{thm:renyi_prop}),  
\begin{align*}
    \Ren_q(\delta_x P_T \mmid \delta_y P_T)
	=
    \frac{\alpha q}{2\,(\exp(2\alpha T) - 1)}\,\|x-y\|^2\,.
\end{align*}
This exactly matches the continuous-time R\'enyi regularity bound in Corollary~\ref{cor:langevin-cont} (and also proved in \S\ref{sec:cont} and \S\ref{app:renyi} below via related approaches).

\subsubsection{Finiteness threshold}\label{app:tightness:finiteness}

Here we provide details for Remark~\ref{rem:finiteness} regarding the finiteness threshold for R\'enyi regularity.

An example where the \emph{second} refined regularity bound in Theorem~\ref{thm:langevin-refined} exactly captures the finiteness threshold (and the unrefined bound in Corollary~\ref{cor:langevin-cont} does not) is the OU process~\eqref{eq:tightness-ou} where $\mu = \delta_0$ and the other initial distribution $\nu$ has bounded second moment but tails that are heavier than sub-Gaussian. 
The second moment condition ensures that the second refined regularity bound in Theorem~\ref{thm:langevin-refined}  is finite for every $T > T_0 = 0$, and thus so is the R\'enyi regularity $\Ren_q(\mu P_T \mmid \nu P_T)$. Whereas this is false for the unrefined bound in Corollary~\ref{cor:langevin-cont} since $\nu$ has non-sub-Gaussian tails.

An example where the \emph{first} refined regularity bound in Theorem~\ref{thm:langevin-refined} exactly captures the finiteness threshold (and the unrefined bound in Corollary~\ref{cor:langevin-cont} does not) is the OU process~\eqref{eq:tightness-ou} with initializations $\mu = \cN(0,\sig^2)$ and $\nu = \delta_0$. To analyze this, we make use of two well-known identities.
The first is a formula for the R\'enyi divergence between two Gaussian distributions with unequal variances~\cite[page 4]{van2014renyi}, which generalizes the isotropic case in Theorem~\ref{thm:renyi_prop}.
The second is a formula for the moment generating function of a chi-squared random variable. Of interest to us are the finiteness conditions in these identities.

\begin{lemma}[R\'enyi divergence between Gaussians]\label{lem:renyi-gaussians-full}
    Denote $\sig_q^2  \deq  (1-q)\, \sig_0^2 + q\,\sig_1^2$. Then for any $q \in (0,1) \cup (1,\infty)$,
    \[
        \Ren_q\bigl(\cN(\mu_0,\sig_0^2) \bigm\Vert \cN(\mu_1, \sig_1^2)\bigr)
        =
        \frac{q\,{(\mu_0 - \mu_1)}^2}{2\sig_q^2} + \frac{1}{1-q} \log \frac{\sig_q}{\sig_0^{1-q} \sig_1^q}
    \]
    if $\sig_q^2 > 0$, and is infinite otherwise.
\end{lemma}

\begin{lemma}[Moment generating function for $\chi^2$]\label{lem:mgf-chi}
    For any $\lambda,\sigma > 0$,
    \[
        \E_{X \sim \cN(0,\sigma^2)} \exp\Bigl( \frac{X^2}{\lambda^2} \Bigr) = \sqrt{\frac{\lambda^2}{\lambda^2 - 2\sigma^2}}
    \]
    if $\lambda^2 > 2\sigma^2$, and is infinite otherwise.
\end{lemma}

We now return to the example. 
A calculation gives $\mu P_T = \cN(0, \exp(-2\alpha T)\, \sig^2 + \frac{1-\exp(-2\alpha T)}{\alpha})$ and $\nu P_T = \cN(0, \frac{1-\exp(-2\alpha T)}{\alpha})$. Thus by Lemma~\ref{lem:renyi-gaussians-full}, $\Ren_q(\mu P_T \mmid \nu P_T) $ is finite if and only if $T > T_0  \deq  \frac{1}{2\alpha} \log(1 + (q-1)\,\alpha\sig^2)$. Now by Lemma~\ref{lem:mgf-chi}, the regularity upper bound~\eqref{eq:langevin-regularity-refined-1}
is again finite if and only if $T > T_0$. In contrast, the unrefined bound in Corollary~\ref{cor:langevin-cont} does not tightly capture the threshold since by  Lemma~\ref{lem:mgf-chi}, it is finite if and only if $T >\tfrac{1}{2\alpha} \log(1 + q\,(q-1)\, \alpha \sig^2)$.

\subsection{R\'enyi divergence bounds via continuous-time arguments}\label{app:renyi}

\subsubsection{Synchronous coupling}

Here, we extend the computation in \S\ref{sec:cont:sync} to R\'enyi divergences $\Ren_q$ for orders $q > 1$.
Recall from \S\ref{sec:cont:sync} the definition of the path measures $\PathA_T$, $\PathAux_T$, and that
\begin{align*}
    \mc E_T
    &= \exp\bigl(M_T - \frac{1}{2}\,{[M,M]}_T\bigr)\,, \qquad M_t \deq -\frac{1}{\sqrt 2}\int_0^t \eta_s\,\langle X_s-Y_s,\D B_s'\rangle\,.
\end{align*}
Recalling Remark~\ref{rmk:sync_cont_technical}, it suffices to bound $\E_{\PathAux_T}[\mc E_T^q]$.
From It\^o's formula, for any local martingale $\widetilde M$, the exponential $\exp(\widetilde M - \frac{1}{2}\,[\widetilde M, \widetilde M])$ is also a local martingale.
Applying this to $\widetilde M \deq qM$, we deduce that $\exp(qM - \frac{q^2}{2} \,[M,M])$ is a non-negative local martingale and hence a supermartingale.
From the almost sure bound~\eqref{eq:sync_cont_dist},
\begin{align*}
    \E_{\PathAux_T}[\mc E_T^q]
    &= \E_{\PathAux_T}\exp\Bigl(qM_T - \frac{q^2}{2}\,{[M,M]}_T + \frac{q\,(q-1)}{2}\,{[M,M]}_T \Bigr) \\
    &= \E_{\PathAux_T}\exp\Bigl(qM_T - \frac{q^2}{2}\,{[M,M]}_T + \frac{q\,(q-1)}{4} \int_0^T \eta_t^2\,\norm{X_t-Y_t}^2 \, \D t \Bigr) \\
    &\le \exp\Bigl(\frac{q\,(q-1)}{4}\int_0^T \eta_t^2 \exp\bigl(-2\alpha t - 2\int_0^t \eta_s \, \D s\bigr)\,\D t\Bigr) \,\E_{\PathAux_T}\exp\bigl(qM_T - \frac{q^2}{2}\,{[M,M]}_T\bigr) \\
    &\le \exp\Bigl(\frac{q\,(q-1)}{4}\int_0^T \eta_t^2 \exp\bigl(-2\alpha t - 2\int_0^t \eta_s \, \D s\bigr)\,\D t\Bigr)\,,
\end{align*}
where the last line follows from the supermartingale property.
Hence,
\begin{align*}
    \Ren_q(\delta_x P_T \mmid \delta_y P_T)
    &\le \frac{1}{q-1} \log \E_{\PathAux_T}[\mc E_T^q]
    \le \frac{q}{4} \int_0^T \eta_t^2 \exp\Bigl(-2\alpha t - \int_0^t \eta_s \, \D s\Bigr)\,\D t\,.
\end{align*}
Note that the problem of choosing ${(\eta_t)}_{t\in [0,T]}$ to minimize this expression leads to the same calculus of variations problem as the one we encountered in \S\ref{sec:cont:sync} (and solved in \S\ref{app:calc}).
It yields
\begin{align*}
    \Ren_q(\delta_x P_t \mmid \delta_y P_T)
    &\le \frac{\alpha q\,\norm{x-y}^2}{2\,(\exp(2\alpha T) - 1)}\,.
\end{align*}

\subsubsection{Wasserstein coupling}

Here, we extend the computation in \S\ref{sec:cont:opt} to R\'enyi divergences $\Ren_q$ for orders $q > 1$. First, we require the following lemma about the derivative in time of a divergence along diffusions. Similar arguments have appeared previously in~\cite{VempalaW19} and~\cite[Lemma 12]{chenetal2022proximalsampler}; we give a proof for completeness.

\begin{lemma}
	Let $\psi : \R_+\to\R_+$ be twice continuously differentiable on $(0,\infty)$.
	Consider the associated divergence
	\begin{align*}
		\msf D_\psi(\mu \mmid \nu)
		&\deq \int \psi\bigl( \frac{\D\mu}{\D\nu}\bigr)\,\D\nu\,.
	\end{align*}
    Suppose that ${(\mu_t)}_{t\ge 0}$, ${(\nu_t)}_{t\ge 0}$ are positive and smooth densities evolving according to the equations
	\begin{align*}
		\partial_t \mu_t
		&= \divergence(\mu_t a_t) + c\, \Delta \mu_t\,, \\
		\partial_t \nu_t
		&= \divergence(\nu_t b_t) + c\,\Delta \nu_t\,.
	\end{align*}
    Here, ${(a_t)}_{t\ge 0}$ and ${(b_t)}_{t\ge 0}$ are families of vector fields on $\R^d$ and $c > 0$.
    Then, it holds that
	\begin{align*}
		\partial_t \msf D_\psi(\mu_t \mmid \nu_t)
		&= -\E_{\mu_t}\bigl\langle \nabla \bigl(\psi' \circ \frac{\mu_t}{\nu_t}\bigr),\; c\,\nabla \log \frac{\mu_t}{\nu_t} + a_t - b_t\bigr\rangle\,.
	\end{align*}
\end{lemma}
\begin{proof}
    Using the evolution equations and integration by parts,
    \begin{align*}
        \partial_t \msf D_\psi(\mu_t \mmid \nu_t)
        &= \int \psi'\bigl( \frac{\mu_t}{\nu_t}\bigr) \,\bigl( \partial_t \mu_t - \frac{\mu_t\,\partial_t \nu_t}{\nu_t}\bigr) + \int \psi\bigl( \frac{\mu_t}{\nu_t}\bigr)\,\partial_t \nu_t \\
        &= -\int \bigl\langle \nabla\bigl[ \psi'\bigl(\frac{\mu_t}{\nu_t}\bigr)\bigr], a_t + c\,\nabla \log \mu_t - b_t - c\,\nabla \log \nu_t \bigr\rangle \, \mu_t \\
        &\qquad{} + {\underbrace{\int \psi'\bigl(\frac{\mu_t}{\nu_t}\bigr)\,\bigl\langle \nabla \frac{\mu_t}{\nu_t}, b_t + c\,\nabla \log \nu_t\bigr\rangle\,\nu_t -\int \bigl\langle \nabla\bigl[\psi\bigl( \frac{\mu_t}{\nu_t}\bigr)], b_t + c\,\nabla \log \nu_t\bigr\rangle \, \nu_t}_{=0}}\,. \qedhere
    \end{align*}
\end{proof}

Consider $\psi = {(\cdot)}^q -1$ for which $\msf D_\psi = \Hell_q$.
Let ${(\mu_t)}_{t\ge 0}$, ${(\nu_t)}_{t\ge 0}$ denote the marginal laws of Langevin diffusions with $\alpha$-convex and smooth potential $V$, and let ${(\mu_t')}_{t\ge 0}$ to be the surrogate process with
\begin{align}\label{eq:cont_ot_aux_renyi}
    \partial_t \mu_t' = \divergence\bigl(\mu_t'\,(\nabla V - \eta_t\,(T_t - {\id}))\bigr) + \Delta \mu_t'\,, \qquad \mu_0' = \nu_0\,.
\end{align}
Here, we take $T_t$ to be a transport map from $\mu_t'$ to $\mu_t$.
Applying the simultaneous diffusion lemma with $\mu_t'$ replacing $\mu_t$ (and ignoring issues of regularity),
\begin{align*}
	\partial_t \Hell_q(\mu_t' \mmid \nu_t)
	&= -q\,\E_{\mu_t'}\bigl\langle \nabla\bigl[\bigl( \frac{\mu_t'}{\nu_t}\bigr){\bigsp}^{q-1}\bigr], \nabla \log \frac{\mu_t'}{\nu_t} - \eta_t\,(T_t - {\id})\bigr\rangle\,.
\end{align*}
We now massage this expression via the chain rule:
\begin{align*}
	\E_{\mu_t'}\bigl\langle \nabla\bigl[\bigl( \frac{\mu_t'}{\nu_t}\bigr){\bigsp}^{q-1}\bigr], \nabla \log \frac{\mu_t'}{\nu_t}\bigr\rangle
	&= (q-1)\,\E_{\mu_t'}\bigl[ \bigl( \frac{\mu_t'}{\nu_t}\bigr){\bigsp}^{q-1}\,\bigl\lVert \nabla \log \frac{\mu_t'}{\nu_t}\bigr\rVert^2\bigr]
\end{align*}
and
\begin{align*}
	&\E_{\mu_t'}\bigl\langle\nabla \bigl[\bigl( \frac{\mu_t'}{\nu_t}\bigr){\bigsp}^{q-1}\bigr], \eta_t\,(T_t-{\id})\bigr\rangle
	= (q-1)\, \E_{\mu_t'}\bigl[ \bigl( \frac{\mu_t'}{\nu_t}\bigr){\bigsp}^{q-1}\, \bigl\langle \nabla \log \frac{\mu_t'}{\nu_t}, \eta_t\,(T_t-{\id})\bigr\rangle\bigr] \\
	&\qquad \le (q-1)\,\Bigl\{ \E_{\mu_t'}\bigl[ \bigl( \frac{\mu_t'}{\nu_t}\bigr){\bigsp}^{q-1}\,\bigl\lVert \nabla \log \frac{\mu_t'}{\nu_t}\bigr\rVert^2\bigr] + \frac{\eta_t^2}{4}\,\E_{\mu_t'}\bigl[ \bigl( \frac{\mu_t'}{\nu_t}\bigr){\bigsp}^{q-1}\,\norm{T_t - {\id}}^2\bigr]\Bigr\}\,,
\end{align*}
hence
\begin{align*}
	\partial_t \Hell_q(\mu_t' \mmid \nu_t)
	&\le \frac{\eta_t^2\, q\,(q-1)}{4}\, \E_{\mu_t'}\bigl[\bigl( \frac{\mu_t'}{\nu_t}\bigr){\bigsp}^{q-1}\,\norm{T_t - {\id}}^2\bigr]\,.
\end{align*}
Now, differentiating the R\'enyi divergence via the chain rule, noting that $\Ren_q = \frac{1}{q-1}\log(1+\Hell_q)$,
\begin{align}\label{eq:cont_ot_renyi}
	\partial_t \Ren_q(\mu_t' \mmid \nu_t)
	&\le \frac{\eta_t^2 q}{4} \, \frac{\E_{\mu_t'}[(\frac{\mu_t'}{\nu_t}){\bigsp}^{q-1}\,\norm{T_t - {\id}}^2]}{\E_{\mu_t'}[(\frac{\mu_t'}{\nu_t}){\bigsp}^{q-1}]}
	\le \frac{\eta_t^2 q}{4} \,\norm{T_t - {\id}}_{L^\infty(\mu_t')}^2\,.
\end{align}
If we choose $T_t$ to be an optimal transport map for the $W_\infty$ metric (c.f.~\cite[\S 3.2]{San15OT}), the right-hand side is bounded by $\frac{\eta_t^2 q}{4}\,W_\infty^2(\mu_t',\nu_t)$.
On the other hand, we might still expect that
\begin{align*}
	\partial_t^+ W_\infty(\mu_t,\mu_t')
	&\le -(\alpha+\eta_t)\,W_\infty(\mu_t,\mu_t')\,,
\end{align*}
where $\partial_t^+$ denotes the upper derivative, and thus
\begin{align}\label{eq:cont_ot_w_inf}
    W_\infty(\mu_t,\mu_t')
    &\le \exp\Bigl(-\alpha t-\int_0^t \eta_s\,\D s\Bigr)\,W_\infty(\mu_0,\nu_0)\,.
\end{align}
Taking this as a given, we show how to conclude the proof.
From~\eqref{eq:cont_ot_renyi} and~\eqref{eq:cont_ot_w_inf},
\begin{align*}
    \Ren_q(\mu_T \mmid \nu_T)
    &= \Ren_q(\mu_T' \mmid \nu_T)
    \le \frac{q\,W_\infty^2(\mu_0,\nu_0)}{4} \int_0^T \eta_t^2 \exp\Bigl(-2\alpha t - \int_0^t \eta_s \, \D s\Bigr) \, \D t
    = \frac{\alpha q\,W_\infty^2(\mu_0,\nu_0)}{2\,(\exp(2\alpha T) - 1)}
\end{align*}
provided that we use the choice of ${(\eta_t)}_{t\in [0,T]}$ derived in \S\ref{app:calc}.

Unfortunately, there are several technical hurdles associated with making these computations rigorous (e.g., justifying~\eqref{eq:cont_ot_w_inf}).
In our view, the simplest way to sidestep these issues is to discretize~\eqref{eq:cont_ot_aux_renyi} via a splitting scheme for the two parts of the dynamics, i.e., the Langevin dynamics and the transport, which leads to the discrete-time approach discussed in \S\ref{sec:disc}.
Hence, we simply refer to the argument therein.

\subsection{Reverse Harnack inequalities via semigroup methods}\label{app:semigp_reverse}

In this section, let ${(P_t)}_{t\ge 0}$ denote the Markov semigroup for a diffusion process on a complete Riemannian manifold $\cM$, with infinitesimal generator $\ms L = \Delta - \langle \nabla V, \nabla \cdot\rangle$, and assume that the curvature-dimension condition $\CD(\alpha,\infty)$ holds.
We recall that this is equivalent to $\Ric + \nabla^2 V \succeq \alpha$.

Let $f$ be a positive function and let ${(x_s)}_{s\in [0,t]}$ denote a curve on $\cM$ with $x_0 = x$ and $x_t = y$.
Let $\phi : \R_{>0} \to \R_{>0}$ be a strictly convex function.
Here, we observe that the semigroup proof of~\cite{Wang1997LSINoncompact}, which we reproduce below, holds verbatim for negative exponents $p < 0$.
Differentiating along the semigroup interpolation,
\begin{align*}
    \partial_s P_s[\phi(P_{t-s} f)](x_s)
    &= P_s[\ms L\phi(P_{t-s} f) - \phi'(P_{t-s} f)\,\ms LP_{t-s} f](x_s) + \langle \nabla [P_s \phi(P_{t-s} f)](x_s), \dot x_s\rangle \\
    &\ge P_s[\phi''(P_{t-s} f)\,\norm{\nabla P_{t-s} f}^2](x_s)- \norm{\nabla [P_s \phi(P_{t-s} f)](x_s)}\,\norm{\dot x_s}\,,
\end{align*}
where we used the diffusion chain rule
\begin{align}\label{eq:diffusion_chain_rule}
    \ms L \phi(f) = \phi'(f) \, \ms L f + \phi''(f) \,\norm{\nabla f}^2\,.
\end{align}
Also, it is well-known that $\CD(\alpha,\infty)$ entails the gradient bound $\norm{\nabla P_s f} \le \exp(-\alpha s)\,P_s\norm{\nabla f}$ for all $s \ge 0$, see Theorem~\ref{thm:equiv_cd}.
Applying this, we obtain
\begin{align*}
    \partial_s P_s[\phi(P_{t-s} f)](x_s)
    &\ge P_s\bigl[\phi''(P_{t-s} f)\,\norm{\nabla P_{t-s} f}^2 - \norm{\dot x_s} \exp(-\alpha s)\,\norm{\nabla \phi(P_{t-s} f)}\bigr](x_s) \\
    &= P_s\bigl[\phi''(P_{t-s} f)\,\norm{\nabla P_{t-s} f}^2 - \norm{\dot x_s} \exp(-\alpha s)\,\abs{\phi'(P_{t-s} f)}\,\norm{\nabla P_{t-s} f}\bigr](x_s) \\
    &\ge -\frac{\norm{\dot x_s}^2 \exp(-2\alpha s)}{4}\, P_s\Bigl(\frac{{\phi'(P_{t-s} f)}^2}{\phi''(P_{t-s} f)}\Bigr)(x_s)\,.
\end{align*}
We now specialize this computation to the case when $\phi(\cdot) = {(\cdot)}^p$, $p \in (-\infty, 0)$, which is strictly convex and positive on $\R_{>0}$, for which ${\phi'(z)}^2/\phi''(z) = \frac{p}{p-1}\,\phi(z)$.
Hence, by Gr\"onwall's inequality,
\begin{align*}
    P_t(f^p)(y)
    &\ge {(P_t f(x))}^p \exp\Bigl( -\frac{\abs p}{4\,\abs{p-1}}\int_0^t \norm{\dot x_s}^2 \exp(-2\alpha s) \, \D s \Bigr)\,.
\end{align*}
Finally, if $\gamma : [0,1] \to \cM$ denotes the constant-speed geodesic joining $x$ to $y$, we set
\begin{align*}
    x_s \deq \gamma\Bigl(\frac{\exp(2\alpha s)-1}{\exp(2\alpha t) - 1}\Bigr)\,, \qquad \norm{\dot x_s} = \frac{2\alpha \exp(2\alpha s)}{\exp(2\alpha t) - 1}\,\msf d(x,y)\,,
\end{align*}
which yields
\begin{align}\label{eq:reverse_harnack_app}
    P_t(f^p)(y)
    &\ge {(P_t f(x))}^p \exp\Bigl( -\frac{\alpha\,\abs p}{2\,\abs{p-1}\,(\exp(2\alpha t) - 1)}\,{\msf d(x,y)}^2\Bigr)\,.
\end{align}
Applying the decreasing function ${(\cdot)}^{1/p}$ to both sides of this inequality yields the reverse Harnack inequality with exponent $p\in(-\infty, 0)$.

We now argue that~\eqref{eq:reverse_harnack_app} implies back the $\CD(\alpha,\infty)$ condition; the proof is similar to the one for~\cite[Theorem 2.3.3]{Wang14Diffusion}.
Let $f : \cM \to \R$ be bounded, smooth, and constant outside of a compact set, with $f\ge C > 0$, $\nabla f(x) = v \in T_p\cM$, and $\nabla^2 f(x) = 0$.
For $t \ge 0$, let $x_t \deq \exp_x(ct\,\nabla \log f(x))$, where $c \in\R$ is to be chosen later.
Then,~\eqref{eq:reverse_harnack_app} readily implies
\begin{align}\label{eq:log_rev_harnack}
    \log P_t(f^p)(x)
    &\ge p\log P_t f(x_t) - \frac{\alpha c^2 p t^2}{2\,(p-1)\,(\exp(2\alpha t) - 1)}\,\norm{\nabla \log f(x)}^2\,.
\end{align}
On one hand,
\begin{align*}
    \partial_t \log P_t(f^p)(x)
    &= \frac{\ms LP_t(f^p)(x)}{P_t(f^p)(x)}\,,
\end{align*}
and hence, a tedious calculation using repeated applications of the diffusion chain rule~\eqref{eq:diffusion_chain_rule}, the product rule $\ms L(fg) = f\,\ms Lg + g\,\ms L f + 2\,\langle \nabla f, \nabla g\rangle$, and $\nabla^2 f(x) = 0$ yields
\begin{align*}
    \partial_t\big|_{t=0} \log P_t(f^p)(x)
    &= \frac{\ms L(f^p)(x)}{{f(x)}^p}
    = \frac{p\,\ms Lf(x)}{f(x)} + \frac{p\,(p-1)\,\norm{\nabla f(x)}^2}{{f(x)}^2}\,, \\
    \partial_t^2\big|_{t=0} \log P_t(f^p)(x)
    &= \frac{\ms L^2(f^p)(x)}{{f(x)}^p} - \frac{{\ms L(f^p)(x)}^2}{{f(x)}^{2p}} \\
    &= \frac{p\,\ms L^2 f(x)}{f(x)}-\frac{p\,{\ms Lf(x)}^2}{{f(x)}^2} + \frac{p\,(p-1)\,\ms L(\norm{\nabla f}^2)(x)}{{f(x)}^2} \\
    &\qquad{} + \frac{2p\,(p-1)\,\langle \nabla \ms Lf(x),\nabla f(x)\rangle}{{f(x)}^2} -\frac{4p\,(p-1)\,\ms Lf(x)\,\norm{\nabla f(x)}^2}{{f(x)}^3} \\
    &\qquad{}- \frac{2p\,(p-1)\,(2p-3)\,\norm{\nabla f(x)}^4}{{f(x)}^4}\,.
\end{align*}
On the other hand,
\begin{align*}
    \partial_t[p\log P_t f(x_t)]
    &= \frac{p\,[\ms LP_t f(x_t) + \langle \nabla P_t f(x_t), \dot x_t \rangle]}{P_t f(x_t)}
\end{align*}
and hence, another tedious calculation yields
\begin{align*}
    \partial_t\big|_{t=0} [p\log P_t f(x_t)]
    &= \frac{p\,\ms Lf(x)}{f(x)} + \frac{cp\,\norm{\nabla f(x)}^2}{{f(x)}^2}\,, \\
    \partial_t^2\big|_{t=0} [p\log P_t f(x_t)]
    &= \frac{p\,[\ms L^2 f(x) + 2c\,\langle \nabla \ms L f(x), \nabla \log f(x) \rangle]}{f(x)} - \frac{p\,{[\ms Lf(x) + c\,\norm{\nabla f(x)}^2/f(x)]}^2}{{f(x)}^2} \\
    &= \frac{p\,\ms L^2 f(x)}{f(x)} - \frac{p\,{\ms Lf(x)}^2}{{f(x)}^2} + \frac{2cp\,\langle \nabla \ms L f(x),\nabla f(x)\rangle}{{f(x)}^2} \\
    &\qquad{} - \frac{2cp\,\ms Lf(x)\,\norm{\nabla f(x)}^2}{{f(x)}^3}
    - \frac{c^2 p\,\norm{\nabla f(x)}^4}{{f(x)}^4}\,.
\end{align*}
We now set $c = 2\,(p-1)$, substitute these expansions into~\eqref{eq:log_rev_harnack}, and divide by $t^2$, obtaining
\begin{align*}
    &\frac{p\,(p-1)}{t}\,\Bigl( \frac{2\alpha t}{\exp(2\alpha t)-1} - 1\Bigr)\,\norm{\nabla \log f(x)}^2 \\
    &\qquad \ge \frac{1}{2}\,\Bigl(-\frac{p\,(p-1)\,[\ms L(\norm{\nabla f}^2)(x) - 2\,\langle \nabla \ms L f(x),\nabla f(x)\rangle]}{{f(x)}^2} - \frac{4p\,{(p-1)}^2\,\norm{\nabla f(x)}^4}{{f(x)}^4}\Bigr) - o(1)\,.
\end{align*}
Sending $t\searrow 0$ and using $\Gamma_2(f,f) = \frac{1}{2}\,(\ms L(\norm{\nabla f}^2) - 2\,\langle \nabla \ms Lf, \nabla f\rangle)$, we deduce that
\begin{align*}
    -\alpha p\,(p-1)\,\norm v^2
    &\ge -p\,(p-1)\,\Gamma_2(f,f)(x) - \frac{2p\,{(p-1)}^2\,\norm v^4}{{f(x)}^2}\,.
\end{align*}
The Bochner{--}Weitzenb\"{o}ck formula yields $\Gamma_2(f,f)(x) = \Ric_x(v,v) + \langle \nabla^2 V(x)\,v, v\rangle$, so that
\begin{align*}
    \Ric_x(v,v) + \langle \nabla^2 V(x) \,v, v\rangle
    &\ge \alpha\,\norm v^2 - \frac{2\,(p-1)\,\norm v^4}{{f(x)}^2}\,.
\end{align*}
Since $f \ge C$ and we can freely send $C \to\infty$, it follows that
\begin{align*}
    \Ric_x(v,v) + \langle \nabla^2 V(x)\,v, v\rangle \ge \alpha\,\norm v^2
\end{align*}
for all $x\in \cM$ and $v \in T_x \cM$, which is $\CD(\alpha,\infty)$.
  	\section{Dual proofs}\label{app:dual}

In order to emphasize the duality between Harnack inequalities and reverse transport inequalities discussed in \S\ref{ssec:apps:background}, in this section we record dual versions of the proofs. Namely, if a fact was established for Harnack inequalities, then here we prove via direct means the corresponding fact for reverse transport inequalities, and vice versa.

\subsection{Distributional Harnack inequalities}\label{app:dist_harnack}

Here we show how to obtain the distributional Harnack inequalities~\eqref{eq:dist_power_harnack} and~\eqref{eq:dist_log_harnack} from standard Harnack inequalities without dualizing (to obtain a reverse transport inequality), appealing to the convexity principle, and dualizing back as described in Remark~\ref{rem:distributional-harnack}. For a function $f > 0$, suppose
\begin{align}\label{eq:general_power_harnack}
    Pf(x) \le C_p(x,y)\, {P(f^p)(y)}^{1/p}\,, \qquad \text{for all}~x,y\in\cX\,.
\end{align}
Integrating this inequality w.r.t.\ $\gamma(\D x, \D y)$, where $\gamma \in \Coup(\mu,\nu)$, and applying H\"older's inequality,
\begin{align*}
    \norm f_{L^1(\mu P)} \le \int C_p(x,y)\, {P(f^p)(y)}^{1/p}\, \gamma(\D x,\D y)
    \le \norm f_{L^p(\nu P)}\,\norm{C_p}_{L^{p/(p-1)}(\gamma)}\,.
\end{align*}
This yields~\eqref{eq:dist_power_harnack}.
The proof for~\eqref{eq:dist_log_harnack} is similar, using Jensen's inequality in lieu of H\"older.

We can refine this inequality as follows. If we integrate over $x$ first, then
\begin{align*}
    \norm f_{L^1(\mu P)}
    &\le \int \Bigl[\int C_p(x,y)\,\gamma_{1\mid 2}(\D x \mid y)\Bigr] \, {P(f^p)(y)}^{1/p}\,\nu(\D y) \\
    &\le \norm f_{L^p(\nu P)}\,\Bigl\{\int \Bigl[\int C_p(x,y) \,\gamma_{1\mid 2}(\D x \mid y)\Bigr]^{p/(p-1)} \, \nu(\D y)\Bigr\}^{(p-1)/p}\,.
\end{align*}
This is equivalent to the first refined R\'enyi bound~\eqref{eq:refined_1} via the duality in \S\ref{ssec:apps:duality}.

To dualize~\eqref{eq:refined_2}, we take the logarithm of~\eqref{eq:general_power_harnack} to obtain
\begin{align*}
    \log Pf(x)
    &\le \frac{1}{p} \log P(f^p)(y) + \log C_p(x,y)\,.
\end{align*}
We integrate w.r.t.\ $\gamma_{2\mid 1}(\D y \mid x)$ and exponentiate to obtain
\begin{align*}
    Pf(x)
    &\le \exp \int\bigl(\frac{1}{p} \log P(f^p)(y) + \log C_p(x,y)\bigr)\,\gamma_{2\mid 1}(\D y \mid x)\,.
\end{align*}
Integrating w.r.t.\ $\mu(\D x)$ and applying H\"older's and Jensen's inequalities,
\begin{align*}
    \norm f_{L^1(\mu P)}
    &\le \Bigl[ \int \exp\Bigl(\int \log P(f^p)(y)\,\gamma_{2\mid 1}(\D y \mid x)\Bigr)\,\mu(\D x) \Bigr]^{1/p} \\
    &\qquad{} \times \Bigl[\int \exp\Bigl( \frac{p}{p-1}\int \log C_p(x,y)\,\gamma_{2\mid 1}(\D y \mid x)\Bigr)\,\mu(\D x) \Bigr]^{(p-1)/p} \\
    &\le \norm f_{L^p(\nu P)}\,\Bigl[\int \exp\Bigl( \frac{p}{p-1}\int \log C_p(x,y)\,\gamma_{2\mid 1}(\D y \mid x)\Bigr)\,\mu(\D x) \Bigr]^{(p-1)/p}\,.
\end{align*}
This is equivalent to the second refined R\'enyi bound~\eqref{eq:refined_2} via the duality in \S\ref{ssec:apps:duality}.

\subsection{Composition of reverse transport inequalities}\label{app:renyi_implies_kl}

In~\cite{Wang10HarnackBoundary}, F.-Y.\ Wang proved that if the power Harnack inequality~\eqref{eq:power_harnack} holds for exponents $p_0, p_1 > 1$, then it also holds for exponent $p_0 p_1$.
Here, we state and prove the dual version of this statement, which shows in particular that any R\'enyi reverse transport inequality of order $q > 1$ implies the corresponding sharp KL reverse transport inequality.
Our proof is based on the \emph{weak triangle inequality} for R\'enyi divergences, which in turn follows from H\"older's inequality (see e.g.,~\cite[Proposition 11]{mironov2017renyi}).

\begin{lemma}[Weak triangle inequality for R\'enyi divergence]
    For any $q > 1$, any $\lambda \in [0,1]$, and any probability measures $\mu$, $\nu$, and $\xi$,
    \begin{align*}
        \Ren_q(\mu \mmid \nu)
        &\le \frac{q-\lambda}{q-1} \,\Ren_{q/\lambda}(\mu \mmid \xi) + \Ren_{(q-\lambda)/(1-\lambda)}(\xi \mmid \nu)\,.
    \end{align*}
\end{lemma}

\begin{lemma}
    Let $P$ be a Markov kernel on a geodesic space $(\cX, \msf d)$.
    Consider the following reverse transport inequality:
    \begin{align}\label{eq:general_rev_transport}
        \Ren_q(\delta_x P \mmid \delta_y P) \le Cq\,{\msf d(x,y)}^2\,, \qquad \text{for all}~x,y\in\cX\,.
    \end{align}
    If~\eqref{eq:general_rev_transport} holds for $q \in \{q_0, q_1\}$, where $q_0, q_1 > 1$, then it also holds for $q= q_0 q_1/(q_0+q_1-1)$.
\end{lemma}

In particular, if~\eqref{eq:general_rev_transport} holds for some order $q > 1$, then it also holds for order $q^2/(2q-1) < q$.
Iterating this, it follows that it holds for $q=1$ (the KL reverse transport inequality).

\begin{proof}
    Let $q \deq q_0 q_1/(q_0 +q_1-1)$.
    Let $\lambda = \frac{q}{q_0} = \frac{q_1}{q_0+q_1-1}$, so that $1-\lambda = \frac{q_0-1}{q_0+q_1-1}$.
    The weak triangle inequality and~\eqref{eq:general_rev_transport} imply
    \begin{align*}
        \Ren_q(\delta_x P \mmid \delta_y P)
        &\le \frac{q_1}{q_1-1} \,\Ren_{q_0}(\delta_x P \mmid \delta_z P) + \Ren_{q_1}(\delta_z P \mmid \delta_y P)
        \le Cq_1\,\Bigl( \frac{q_0}{q_1-1}\,{\msf d(x,z)}^2 + {\msf d(z,y)}^2\Bigr)\,.
    \end{align*}
    Since $(\cX, \msf d)$ is a geodesic space, we can choose $z$ such that $\msf d(x,z) = t\,\msf d(x,y)$ and $\msf d(z,y) = (1-t)\,\msf d(x,y)$ where $t = \frac{q_1-1}{q_0+q_1-1} \in [0,1]$.
    With this choice, the right-hand side becomes $Cq\,{\msf d(x,y)}^2$.
\end{proof}

	\small
	\addcontentsline{toc}{section}{References}
	%\bibliographystyle{plainnat}
	%\bibliography{scp1}{}
        \printbibliography{}

\end{document}